\documentclass[a4paper]{amsart}
\usepackage{a4wide}
\usepackage{graphicx}
\usepackage{amsmath}
\usepackage{amssymb}
\usepackage{amsfonts}
\usepackage{amsthm}
\usepackage{eucal}
\usepackage{tikz}
\usetikzlibrary{trees, positioning, calc, cd}
\usepackage{hyperref}
\usepackage[capitalize,noabbrev]{cleveref}

\theoremstyle{plain}
\newtheorem{theorem}{Theorem}[section]

\newtheorem{lemma}[theorem]{Lemma}
\newtheorem{proposition}[theorem]{Proposition}

\newtheorem{corollary}[theorem]{Corollary}

\newtheorem*{conjecture}{Triviality Conjecture}

\newtheorem*{proposition*}{Proposition}
\newtheorem*{corollary*}{Corollary}
\newtheorem*{theorem*}{Theorem}

\theoremstyle{plain}
\newcounter{zaehler}

\newtheorem{introthm}[zaehler]{Theorem}

\newtheorem*{introcor*}{Corollary}

\theoremstyle{definition}
\newtheorem{definition}[theorem]{Definition}
\newtheorem{example}[theorem]{Example}

\newtheorem{remark}[theorem]{Remark}
\newtheorem{observation}[theorem]{Observation}

\newcommand{\E}{\mathbb{E}}
\newcommand{\Z}{\mathbb{Z}}
\newcommand{\F}{\mathbb{F}}
\newcommand{\HH}{\mathrm{HH}}

\newcommand{\C}{\mathcal{C}}
\newcommand{\D}{\mathcal{D}}
\newcommand{\lto}{\longrightarrow}
\newcommand{\KU}{\mathrm{KU}}
\newcommand{\free}{\mathrm{free}}
\newcommand{\Barc}{\mathrm{Bar}}
\newcommand{\Alg}{\mathrm{Alg}}

\newcommand{\res}{\mathrm{res}}

\newcommand{\Spc}{\mathcal{S}}
\newcommand{\Map}{\mathrm{Map}}
\newcommand{\map}{\mathrm{map}}
\renewcommand{\nu}{\mathrm{nu}}
\newcommand{\fgt}{\mathrm{fgt}}
\newcommand{\m}{\mathfrak{m}}
\newcommand{\op}{\mathrm{op}}

\newcommand{\Cat}{\mathrm{Cat}_{\infty}}
\renewcommand{\O}{\mathcal{O}}

\newcommand{\SSeq}{\mathrm{SSeq}}
\newcommand{\End}{\mathrm{End}}
\newcommand{\id}{\mathrm{id}}

\newcommand{\fib}{\mathrm{fib}}
\newcommand{\gr}{\mathrm{gr}}

\newcommand{\Mod}{\mathrm{Mod}}
\newcommand{\Sp}{\mathrm{Sp}}
\newcommand{\tr}{\mathrm{tr}}
\newcommand{\bS}{\mathbb{S}}
\newcommand{\aug}{\mathrm{aug}}
\newcommand{\RMod}{\mathrm{RMod}}

\newcommand{\wc}{\mathrm{wc}}

\newcommand{\Monad}{\mathrm{Monad}}
\newcommand{\Operad}{\mathrm{Operad}}
\newcommand{\Fun}{\mathrm{Fun}}

\newcommand{\Fin}{\mathrm{Fin}}
\newcommand{\ind}{\mathrm{ind}}
\newcommand{\R}{\mathbb{R}}
\newcommand{\RP}{\mathbb{RP}}

\newcommand{\BT}{\mathrm{B}\mathbb{T}}
\newcommand{\T}{\mathbb{T}}

\newcommand{\Sym}{\mathrm{free}}
\newcommand{\one}{\mathbf{1}}
\newcommand{\Equiv}{\mathrm{Equiv}}

\renewcommand{\P}{\mathcal{P}}
\newcommand{\Q}{\mathbb{Q}}

\makeatletter
\@addtoreset{theorem}{section}
\@addtoreset{subsection}{section}
\makeatother

\begin{document}

\title{Formality of $\E_n$-algebras and cochains on spheres}

\thanks{ML was supported by the Danish National Research Foundation through the Copenhagen Centre for Geometry and Topology (DNRF151). GH was supported by an ERC Starting Grant (no.\ 950048) and an NWO VIDI grant (no.\ 223.093).} 

\author[G.~Heuts]{Gijs Heuts} 
\address{Mathematical Institute, Utrecht University, Budapestlaan 6, 3584 CD Utrecht, The Netherlands}
\email{g.s.k.s.heuts@uu.nl}

\author[M.~Land]{Markus Land}
\address{Mathematisches Institut, Ludwig-Maximilians-Universit\"at M\"unchen, Theresienstra\ss e 39, 80333 M\"unchen, Germany}
\email{markus.land@math.lmu.de}

\date{\today}

\begin{abstract}
We study the loop and suspension functors on the category of augmented $\E_n$-algebras. One application is to the formality of the cochain algebra of the $n$-sphere. We show that it is formal as an $\E_n$-algebra, also with coefficients in general commutative ring spectra, but rarely $\E_{n+1}$-formal unless the coefficients are rational. Along the way we show that the free functor from operads in spectra to monads in spectra is fully faithful on a nice subcategory of operads which in particular contains the stable $\E_n$-operads for finite $n$. We use this to interpret our results on loop and suspension functors of augmented algebras in operadic terms.
\end{abstract}

\maketitle 

\tableofcontents

\section{Introduction}

The notion of formality originates in rational homotopy theory, where a space $X$ is called \emph{formal} if its $\E_{\infty}$-algebra $C^*(X;\Q)$ of rational cochains is equivalent, as an $\E_\infty$-algebra, to the cohomology ring $H^*(X;\Q)$.\footnote{Classically, this is formulated via Sullivan's $\mathcal{A}_{\mathrm{PL}}$ functor as a zig-zag of quasi-isomorphisms of CDGA's.} There are many interesting examples, including spheres, loop spaces, and compact K\"ahler manifolds. One surprising aspect of rational formality of a space $X$ is that it suffices to show that $C^*(X;\Q)$ and $H^*(X;\Q)$ are equivalent as $\E_1$-algebras \cite{saleh,CPRNW}.

If one changes coefficients from $\Q$ to a finite field $\F_p$, the situation changes drastically. One of the simplest examples, namely the $\E_{\infty}$-algebra $C^*(S^n;\F_p)$ of cochains on a sphere, already fails to be formal. Indeed, the cohomology of a sphere is a trivial square-zero algebra, but the $\E_\infty$-structure of the cochain algebra is not trivial. The latter is essentially a consequence of the fact that the operation $\mathrm{Sq}^0$ acts nontrivially on $H^n(S^n;\F_p)$. Indeed, the Steenrod squares can be interpreted as power operations which exist for any $\E_\infty$-algebra over $\F_p$ and vanish on trivial algebras (see \cref{prop:DL-operation} for a more general result).

Still, one can ask for a weaker notion of formality: does there exist a $k \geq 1$ such that $C^*(S^n;\F_p)$ and $H^*(S^n;\F_p)$ are equivalent as $\E_k$-algebras? The answer turns out to be yes if $k \leq n$, independently of the coefficients. In fact, we will prove a stronger statement. If $R$ is an $\E_{n+1}$-ring and $X$ a space, we write $C^*(X;R)$ for the \emph{$R$-valued cochains on $X$}, defined as the limit $\varprojlim_X R$ in the $\infty$-category $\Alg_{\E_n}(\Mod_R)$ of $\E_n$-$R$-algebras. A choice of basepoint in $X$ equips $C^*(X;R)$ with an augmentation to $R$, making it into an object of $\Alg^\aug_{\E_n}(\Mod_R)$. We shall refer to a trivial square-zero extension in $\Alg_{\E_n}^\aug(\Mod_R)$ as an algebra which is \emph{$\E_n$-trivial over $R$}, see \cref{section:stable-case} for the precise definitions.

\begin{introthm}
\label{thm:intro}
Let $R$ be an $\E_{n+1}$-ring spectrum and $X$ a pointed space. Then $C^*(\Sigma^n X;R)$ is $\E_n$-trivial over $R$.
\end{introthm}

Specializing to the case $X = S^0$ and $R = \F_p$ implies the $\E_n$-formality of $C^*(S^n;\F_p)$. In fact, \cref{thm:intro} itself is a special case of a much more general statement about augmented $\E_n$-algebras in stable $\infty$-categories:

\begin{introthm}
\label{thm:intro2}
Let $\C$ be a stably $\E_n$-monoidal $\infty$-category and $A$ an augmented $\E_n$-algebra in $\C$. Then the $n$-fold loop object $\Omega^n A$ is $\E_n$-trivial.
\end{introthm}
Here the loops functor $\Omega$ is computed in the $\infty$-category $\Alg_{\E_n}^{\aug}(\C)$ of augmented $\E_n$-algebras in $\C$ and again, we refer to \cref{section:stable-case} for the precise definitions of trivial $\E_n$-algebras in $\C$.
\cref{thm:intro} is the special case $\C = \mathrm{Mod}_R$ and $A=C^*(X;R)$, using that $C^*(\Sigma^n X; R) \cong \Omega^n C^*(X;R)$. Let us remark that Mandell has considered the same question and claims a version of \cref{thm:intro} in the case where $R$ is an ordinary commutative ring \cite{mandell}, but we are not aware of a published reference of his, and that work of Berger--Fresse \cite{Berger-Fresse} proves \cref{thm:intro} again in the case where $R$ is an ordinary commutative ring.

We also investigate to what extent \cref{thm:intro} is sharp and first conjecture the following.

\begin{conjecture}\hypertarget{conj}
Let $n \geq 1$ and let $A$ be an $\E_{n+2}$-algebra in $\Sp$. Then $C^*(S^n;A)$ is $\E_{n+1}$-trivial over $A$ if and only if $A$ is a $\Q$-algebra.
\end{conjecture}
The interesting part of this conjecture is the \emph{only if} implication.
We offer a proof in many cases of interest:

%

\begin{introthm}[See Theorems~\ref{thm:conj} and \ref{thm:conj2}]\label{thm:conj-intro}
The \hyperlink{conj}{Triviality Conjecture} holds 
\begin{enumerate}
\item\label{item:thm1} when $n\leq 2$,
\item\label{item:thm2} when $A$ is bounded below, 
\item\label{item:thm3} when $A$ is the underlying $\E_{n+2}$-algebra of an $\E_\infty$-algebra, and
\item\label{item:thm4} when there exists a prime $p$ such that $A \otimes \F_p$ or $A \otimes \KU/p$ are non-zero.
\end{enumerate}
\end{introthm}

We use two approaches to prove \cref{thm:conj-intro}. One is based on studying the operadic structure maps of $C^*(S^n;A)$ and the other is based on power operations. The cases above are not unrelated, indeed \ref{item:thm4} implies \ref{item:thm2} and \ref{item:thm3}, see \cref{lemma:implication}.

Let us now outline our approach to proving \cref{thm:intro2}. The key idea is to show that the $\E_n$-algebra structure of $\Omega A$ only depends on the $\E_{n-1}$-structure rather than the full $\E_n$-structure of $A$. More precisely, we show in \cref{cor:loopsEn} that there exists a commutative diagram as follows.
\[
\begin{tikzcd}
\Alg_{\E_n}^{\mathrm{aug}}(\C) \ar{dr}{\Omega} \ar{d}[swap]{\res_{\E_{n-1}}^{\E_n}}\ar{r}{\omega} & \Alg_{\E_{n+1}}^{\mathrm{aug}}(\C) \ar{d}{\res_{\E_{n}}^{\E_{n+1}}} \\
\Alg_{\E_{n-1}}^{\mathrm{aug}}(\C) \ar{r}[swap]{\omega} & \Alg_{\E_n}^{\mathrm{aug}}(\C).
\end{tikzcd}
\]
Here the vertical functors are the evident forgetful ones. The horizontal functors $\omega$ are less obvious; they can be constructed as the right adjoint to the bar construction. The lower triangle establishes the previously mentioned key idea, and the upper triangles then allows us to inductively deduce that $\Omega^n$ is equivalent to $\omega^n \circ \res^{\E_n}_{\E_0}$. The proof of \cref{thm:intro2} then boils down to relating $\omega^n$ with the trivial algebra functor, which we do in \cref{thm:trivial-algebras}.

We will offer an alternative interpretation of this square in terms of the underlying operads. In Section \ref{sec:stableEn} we show that the existence of the diagram above is essentially equivalent to a diagram of operads in the $\infty$-category of spectra as follows:
\[
\begin{tikzcd}
\Sigma^\infty_+ \E_n^{\mathrm{nu}} \ar{r}\ar{d}\ar{dr}{\sigma} & \Sigma^\infty_+ \E_{n+1}^{\mathrm{nu}} \ar{d} \\
S \Sigma^\infty_+ \E_{n-1}^{\mathrm{nu}} \ar{r} & S \Sigma^\infty_+ \E_n^{\mathrm{nu}}.
\end{tikzcd}
\]
Here the $S$ indicates \emph{operadic suspension} and $\sigma$ denotes the \emph{operadic suspension map}; we discuss both of these concepts in \cref{sec:stableEn} as well. The surprising result is that these two squares are essentially equivalent: passing from the operadic square to a square of algebras is immediate, but the converse is not. Again, we remark that in the $R$-linear context, with $R$ an ordinary commutative ring, a square of operads as above has been constructed by Fresse \cite[Theorem C]{Fresse} using point-set models. In particular his results also provide a proof of \cref{thm:intro2} for $R$-linear categories.

To pass between the two diagrams above, we prove the following result about the relation between operads and monads that could be of independent interest. 

\begin{introthm}
\label{thm:intro3}
The functors 
\[\free\colon \SSeq(\Sp) \lto \End(\Sp) \quad \text{ and } \quad \free \colon \Operad(\Sp) \lto \Monad(\Sp) \]
assigning to a non-unital symmetric sequence (or operad) $\O$ in the $\infty$-category of spectra its corresponding free algebra endo-functor (or monad) $\free_{\O}$ restricts to a fully faithful embedding on the full subcategory containing the non-unital $\E_n$-operads, for $n<\infty$, and their operadic (de)suspensions.
\end{introthm}
In fact, we single out a property of symmetric sequences, which we call having \emph{nilpotent Euler classes} and show that the free functor is fully faithful on symmetric sequences and operads with nilpotent Euler classes, and that the $\E_n$-operads as well as their operadic (de)suspensions have nilpotent Euler classes, see \cref{thm:SymffEuler}.

Finally, we outline the plan of this paper. In \cref{sec:susploopEn} we prove the basic results on loops and suspensions of $\E_n$-algebras, in particular establishing \cref{thm:intro2} and a dual version for $n$-fold suspensions. In \cref{sec:stableEn} we investigate the relation between operads and monads on the $\infty$-category $\Sp$ and establish \cref{thm:intro3}. The final \cref{sec:trivspheres} is focused on the cochain algebras of spheres and is mostly devoted to the inspection of our \hyperlink{conj}{Triviality Conjecture}.

\section{Suspensions and loops of $\E_n$-algebras}
\label{sec:susploopEn}
\subsection{Preliminaries}
Let us denote by $\Cat^\wc$ the $\infty$-category of $\infty$-categories which admit weakly contractible colimits and functors which preserve weakly contractible colimits. Via Lurie's tensor product \cite[\S 4.8.1]{HA}, $\Cat^\wc$ acquires a symmetric monoidal structure. Throughout this section, we let $\C$ be an object of $\Alg_{\E_n}(\Cat^\wc)$ and sometimes refer to it as an $\E_n$-monoidal category with weakly contractible colimits. By definition, the tensor product of $\C$ preserves weakly contractible colimits in each variable. We denote the monoidal unit in $\C$ by $\one$ and write $\Alg_{\E_n}^\aug(\C) = \Alg_{\E_n}(\C)_{/\one}$ for the category of augmented $\E_n$-algebras. This is a pointed category, with zero object given by the monoidal unit $\one$ of $\C$. Moreover, Dunn additivity \cite[Theorem 5.1.2.2]{HA} induces the equivalence $\Alg_{\E_n}^{(\aug)}(\C) \simeq \Alg_{\E_1}^{(\aug)}(\Alg_{\E_{n-1}}(\C))$.
We will need the following lemma.
\begin{lemma}\label{lemma:algebras}
Let $n\geq 1$ and $\C$ be in $\Alg_{\E_n}(\Cat^\wc)$ and $0< k<n$. Then $\Alg_{\E_{k}}(\C)$ is in $\Alg_{\E_{n-k}}(\Cat^\wc)$. Moreover, the functor $\res_{\E_{n-1}}^{\E_n}$ admits a left adjoint $\free_{\E_{n-1}}^{\E_n} \colon  \Alg_{\E_{n-1}}^{\aug}(\C) \rightarrow \Alg_{\E_{n}}^{\aug}(\C)$.
\end{lemma}
\begin{proof}
The first statement is \cite[Prop.\ 5.1.2.9]{HA}. For the second, \cite[Remark 5.2.2.10]{HA} then implies that the forgetful functor $\Alg_{\E_1}(\Alg_{\E_{n-1}}(\C)) \to \Alg_{\E_0}(\Alg_{\E_{n-1}}(\C))$ admits a left adjoint. This adjunction induces the required adjunction on augmented objects, using that also $\Alg^\aug_{\E_0}(\Alg_{\E_{n-1}}(\C)) \simeq \Alg^\aug_{\E_{n-1}}(\C)$.
\end{proof}
We note that, as a consequence, $\Alg_{\E_n}^\aug(\C)$ is cocomplete since it admits weakly contractible colimits, as well as an initial object.\footnote{This fact is also true for $n=0$.} In particular, it admits a suspension functor which we denote by $\Sigma_{\E_n}$ and warn the reader that it is not compatible with the forgetful functors $\Alg_{\E_n}^\aug(\C) \to \Alg_{\E_k}^\aug(\C)$ for $k<n$.

\begin{lemma}\label{lemma:Bar}
For $A$ in $\Alg_{\E_n}^\aug(\C)$, the relative tensor product $\one \otimes_A \one$ exists and its formation provides a functor 
\[ \Alg_{\E_n}^\aug(\C) \lto \Alg_{\E_{n-1}}^\aug(\C), \quad A \mapsto \Barc(A) := \one\otimes_A \one \]
called the Bar construction. 
\end{lemma}
\begin{proof}
We note again the equivalence $\Alg_{\E_n}^{(\aug)}(\C) \simeq \Alg_{\E_1}^{(\aug)}(\Alg_{\E_{n-1}}(\C))$ and that $\Alg_{\E_{n-1}}(\C)$ is a monoidal category with weakly contractible colimits by \cref{lemma:algebras}. In particular, $\Alg_{\E_{n-1}}(\C)$ is a monoidal category with geometric realizations, so \cite[\S 4.4.2]{HA} provides a functor $\Barc \colon \Alg_{\E_n}(\C) \to \Alg_{\E_{n-1}}(\C)$ which induces the required functor on augmented objects as clearly $\Barc(\one) \simeq \one$.
\end{proof}

Finally, we will also need the following lemma.
\begin{lemma}\label{lemma:induced-map-tensor-product}
Let $\C$ be an $\E_n$-monoidal category with geometric realizations. Then a commutative square
\[\begin{tikzcd}
	A \ar[r] \ar[d] & B \ar[d] \\
	A' \ar[r] & B'
\end{tikzcd}\]
in $\Alg_{\E_n}^{(\aug)}(\C)$ induces a canonical map $A' \otimes_A B \to B'$ in $\Alg_{\E_{n-1}}^{(\aug)}(\C)$.
\end{lemma}
\begin{proof}
The assumptions on $\C$ imply that the relative tensor product can be formed and that it provides a functor $- \otimes_A B \colon \RMod_A(\C) \to \RMod_B(\C)$ left adjoint to the forgetful functor. This functor in addition inherits the structure of an $\E_{n-1}$-monoidal functor from the $\E_n$-structure on the map $A \to B$. It consequently induces a left adjoint $-\otimes_A B \colon \Alg_{\E_{n-1}}(\RMod_A(\C)) \to \Alg_{\E_{n-1}}(\RMod_B(\C))$. There is therefore an equivalence
\[ \Map_{\Alg_{\E_{n-1}}(\RMod_B(\C))}(A'\otimes_A B,B') \simeq \Map_{\Alg_{\E_{n-1}}(\RMod_A(\C))}(A',B') \]
and the latter space contains the map $A' \to B'$ part of the above diagram. We then use the forgetful functor $\Alg_{\E_{n-1}}(\RMod_B(\C)) \to \Alg_{\E_{n-1}}(\C)$ induced from the unit of $B$ and arrive at the claimed canonical map. This construction is natural in maps of squares, which shows that the just constructed map is compatible with augmentations.
\end{proof}

\subsection{Suspensions and loops of $\E_n$-algebras}
The starting point of our paper is the following result.

\begin{theorem}
\label{thm:suspensionEn}
Let $n\geq 0$ and $\C$ be in $\Alg_{\E_{n+1}}(\Cat^\wc)$. Then there is a natural commutative diagram of colimit preserving functors as follows:
\[
\begin{tikzcd}
\Alg_{\E_n}^{\mathrm{aug}}(\C) \ar{dr}{\Sigma_{\E_n}} \ar{r}{\free_{\E_n}^{\E_{n+1}}}\ar{d}[swap]{\Barc} & \Alg_{\E_{n+1}}^{\mathrm{aug}}(\C) \ar{d}{\Barc} \\
\Alg_{\E_{n-1}}^{\mathrm{aug}}(\C) \ar{r}[swap]{\free_{\E_{n-1}}^{\E_{n}}} & \Alg_{\E_n}^{\mathrm{aug}}(\C)
\end{tikzcd}
\]
Here, we assume $n \geq 1$ for the lower-left triangle.
\end{theorem}
\begin{remark}
The proof we give in fact shows that for the lower-left triangle to commute as claimed, it suffices that $\C$ is $\E_n$-monoidal. 
\end{remark}

\begin{proof}[Proof of \cref{thm:suspensionEn}]
We treat the upper-right triangle first. By \cref{lemma:algebras}, the $\infty$-category $\Alg_{\E_n}(\C)$ lies in $\Alg_{\E_1}(\Cat^\wc)$. 
Using the canonical equivalences $\Alg_{\E_n}^\aug(\C) \simeq \Alg_{\E_n}(\C)_{\mathbf{1} // \mathbf{1}}$ as well as $\Alg_{\E_1}^\aug(\Alg_{\E_n}(\C)) \simeq \Alg^\aug_{\E_{n+1}}(\C)$ for all $n\geq 0$, \cite[Cor.\ 5.2.2.13]{HA} says that the composite
\[ \Alg_{\E_n}^\aug(\C) \xrightarrow{\free_{\E_n}^{\E_{n+1}}} \Alg_{\E_{n+1}}^\aug(\C) \xrightarrow{\Barc} \Alg_{\E_n}(\C) \]
is given by the functor which takes an object $C$ to the pushout $\mathbf{1} \amalg_C \mathbf{1}$. We note that the composite above canonically factors through $\Alg_{\E_n}^\aug(\C)$ since $\Barc(\mathbf{1}) = \mathbf{1}$. The same formula then remains true for this refined functor because the forgetful functor is conservative and preserves pushouts. Since $\Alg_{\E_n}^\aug(\C)$ is a pointed category, with terminal object $\mathbf{1}$, the functor $C \mapsto \mathbf{1} \amalg_C \mathbf{1}$ is the suspension functor $\Sigma_{\E_n}$, showing the commutativity of the upper-right triangle.

Now let us show that all functors in question preserve colimits. First, in any pointed and cocomplete $\infty$-category, the suspension functor commutes with colimits, hence $\Sigma_{\E_n}$ preserves colimits. Moreover, for all $k\geq 1$, the functor $\free_{\E_{k-1}}^{\E_k}$ is a left adjoint and hence also preserves colimits. It then suffices to show that for any $k\geq 1$, the functor $\Barc \colon \Alg^\aug_{\E_k}(\C) \to \Alg^\aug_{\E_{k-1}}(C)$ preserves small colimits. By construction, $\Barc$ preserves sifted colimits, so in fact it suffices to show that it preserves the initial object and binary coproducts. The initial object is given by the tensor unit $\mathbf{1}$ and which is clearly preserved by $\Barc$. Next, we wish to show that for any two objects $A,B$ of $\Alg^\aug_{\E_k}(\C)$, the canonical map 
\[ \Barc(A) \amalg \Barc(A') \to \Barc(A \amalg A') \]
is an equivalence. Since $\Barc$ preserves sifted colimits and every augmented $\E_k$-algebra is a sifted colimit of algebras of the form $\free_{\E_{k-1}}^{\E_k}(B)$ \cite[Prop.\ 4.7.3.14]{HA} (to apply this proposition, note that the forgetful functor $\Alg^\aug_{\E_{k}}(\C) \to \Alg^\aug_{\E_{k-1}}(\C)$ is conservative and preserves sifted colimits), it suffices to show that the canonical map
\[ \Barc(\free_{\E_{k-1}}^{\E_k}(B)) \amalg \Barc(\free_{\E_{k-1}}^{\E_k}(B')) \lto \Barc(\free_{\E_{k-1}}^{\E_k}(B \amalg B')) \]
is an equivalence. By the already established commutativity of the upper-right triangle, this map is equivalent to the canonical map
\[ \Sigma_{\E_k}(B) \amalg \Sigma_{\E_k}(B') \lto \Sigma_{\E_k}(B\amalg B') \]
which is an equivalence since $\Sigma_{\E_k}$ preserves colimits as we have already argued.

Finally, we show that the lower-left triangle commutes. To do so, we first observe that the commutative (in fact pushout) diagram in $\Alg^\aug_{\E_n}(\C)$ 
\[\begin{tikzcd}
	A \ar[r] \ar[d] & \mathbf{1} \ar[d] \\ \mathbf{1} \ar[r] & \Sigma_{\E_n}(A) 
\end{tikzcd}\]
 induces a canonical map $\Barc(A) = \mathbf{1} \otimes_A \mathbf{1} \to \Sigma_{\E_n}(A)$ in $\Alg^\aug_{\E_{n-1}}(\C)$, see \cref{lemma:induced-map-tensor-product} used for the above square for $A$, for $\mathbf{1}$, and the canonical map from the former to the latter. By adjunction, we obtain a natural map $\free_{\E_{n-1}}^{\E_n}(\Barc(A)) \to \Sigma_{\E_n}(A)$. Source and target of this natural map are colimit preserving functors, as we have just argued. The same reasoning as above implies that it therefore suffices to show that this map is an equivalence in case $A = \free_{\E_{n-1}}^{\E_n}(B)$ in which case we may again use the established commutativity of the upper-right triangle for $n-1$ in place of $n$. The map under investigation then becomes equivalent to the canonical map 
\[ \free_{\E_{n-1}}^{\E_n}(\Sigma_{\E_{n-1}}(B)) \lto \Sigma_{\E_n}(\free_{\E_{n-1}}^{\E_n}(B)) \]
which is an equivalence because the free functor preserves colimits and therefore commutes with suspensions.
\end{proof}

Applying Theorem \ref{thm:suspensionEn} several times we obtain the following:

\begin{corollary}
\label{cor:suspensionEn}
Let $\C$ be as in Theorem \ref{thm:suspensionEn} and let $A$ be in $\Alg_{\E_n}^\aug(\C)$. Then there are canonical equivalences 
\[ \Barc^k(\free_{\E_n}^{\E_{n+k}}(A)) \simeq \Sigma^k_{\E_n}(A) \simeq \free_{\E_{n-k}}^{\E_n}(\Barc^k(A))\]
where we assume $k\leq n$ for the right hand equivalence. In particular, the $\E_n$-algebra $\Sigma^n_{\E_n}(A)$ is in the image of the functor $\free_{\E_0}^{\E_n}\colon \Alg_{\E_0}^\aug(\C) \to \Alg_{\E_n}^\aug(\C)$.
\end{corollary}

\begin{corollary}
\label{cor:loopsEn}
Let $\C$ be an $\E_{n+1}$-monoidal $\infty$-category and assume that $\Alg_{\E_0}^\aug(\C)$ admits a loop functor $\Omega$. Then there is a commutative diagram as follows:
\[
\begin{tikzcd}
\Alg_{\E_n}^{\mathrm{aug}}(\C) \ar{dr}{\Omega} \ar{d}[swap]{\res_{\E_{n-1}}^{\E_n}}\ar{r}{\omega} & \Alg_{\E_{n+1}}^{\mathrm{aug}}(\C) \ar{d}{\res_{\E_{n}}^{\E_{n+1}}} \\
\Alg_{\E_{n-1}}^{\mathrm{aug}}(\C) \ar{r}[swap]{\omega} & \Alg_{\E_n}^{\mathrm{aug}}(\C)
\end{tikzcd}
\]
\end{corollary}
\begin{proof}
First, we observe that also the categories $\Alg_{\E_n}^\aug(\C)$ admit a loop functor, since the forgetful functor to $\Alg_{\E_0}^\aug(\C)$ preserves limits and is conservative.
Consider the Yoneda embedding $\C \to \P(\C)$. Day convolution provides an $\E_{n+1}$-monoidal structure on $\P(\C)$ which preserves small colimits in each variable and such that the Yoneda embedding canonically refines to a symmetric monoidal functor \cite[Cor.\ 4.8.1.12]{HA}. Up to enlargening a universe, we may apply \cref{thm:suspensionEn} to $\P(\C)$. In this case, all categories in question are presentable, so that all colimit preserving functors are in fact left adjoints. Passing to right adjoints, we obtain a commutative diagram
\[\begin{tikzcd}
\Alg_{\E_n}^{\mathrm{aug}}(\P(\C)) \ar{dr}{\Omega} \ar{d}[swap]{\res_{\E_{n-1}}^{\E_n}}\ar{r}{\omega} & \Alg_{\E_{n+1}}^{\mathrm{aug}}(\P(\C)) \ar{d}{\res_{\E_{n}}^{\E_{n+1}}} \\
\Alg_{\E_{n-1}}^{\mathrm{aug}}(\P(\C)) \ar{r}[swap]{\omega} & \Alg_{\E_n}^{\mathrm{aug}}(\P(\C))
\end{tikzcd}\]
where $\omega$ denotes the right adjoint of $\Barc$. We claim that this diagram restricts to the one we aim to obtain. For this, we need to see that all functors in question preserve algebras whose underlying object lies in $\C \subseteq \P(\C)$. This is clearly true for the restriction functor and for $\Omega$ as well since the inclusion $\Alg_{\E_n}^\aug(\C) \subseteq \Alg_{\E_n}(\P(\C))$ preserves limits. Finally, we need to show that the composite 
\[\Alg_{\E_n}^\aug(\C) \subseteq \Alg_{\E_n}^\aug(\P(\C)) \stackrel{\omega}{\lto} \Alg_{\E_{n+1}}^\aug(\P(\C)) \lto \P(\C), \]
where the last functor is the forgetful functor, lands in $\C \subseteq \P(\C)$. This follows from the commutativity of the upper-right triangle since the forgetful functor above factors through $\res_{\E_n}^{\E_{n+1}}$, whose composite with $\omega$ is given by $\Omega$, so we can use that $\Omega$ has the desired property.
\end{proof}

\begin{corollary}
\label{cor:loopsE_n}
Let $\C$ be as in \cref{cor:loopsEn} and let $A$ be in $\Alg_{\E_n}^\aug(\C)$. There there are canonical equivalences
\[ \res^{\E_{n+k}}_{\E_n} (\omega^k(A)) \simeq \Omega^k(A) \simeq \omega^k(\res^{\E_n}_{\E_{n-k}}(A))\]
where again, we assume $k\leq n$ for the right hand equivalence.
In particular, the $\E_n$-algebra $\Omega^n(A)$ is in the image of the functor $\omega^n\colon \Alg_{\E_0}^\aug(\C) \to \Alg_{\E_n}^\aug(\C)$.
\end{corollary}

\cref{cor:loopsE_n} provides two useful observations. First, the second equivalence shows that for an augmented $\E_n$-algebra $A$ the $\E_n$-structure of $\Omega A$ only depends on the underlying $\E_{n-1}$-structure of $A$. As mentioned in the introduction, the special case where $A= C^*(X;R)$ is the $\E_\infty$-algebra of $R$-valued cochains on a space $X$, for some commutative ring $R$, has been claimed by Mandell \cite{mandell}. Secondly, the first equivalence shows that the algebra $\Omega A$ admits a natural $\E_{n+1}$-algebra structure (which does, in general, depend on the full $\E_n$-structure of $A$). 

\subsection{The stable case}\label{section:stable-case}
We will be particularly interested in the case where $\C$ is stable, in which case there is an equivalence between augmented and non-unital algebras. Let us recall that a non-unital operad $\O$ is one with $\O(0) = \varnothing$. There are then the non-unital versions $\E_n^\nu$ of the $\E_n$-operads and they satisfy 
\[\E_n^\nu(k) = \begin{cases} \varnothing & \text{ if } n=0 \\ \E_n(k) & \text{ else} \end{cases}\]
and hence come with canonical morphisms $\E_n^\nu \to \E_n$. For $\C$ an $\E_n$-monoidal category, we write $\Alg_{\E_n}^\nu(\C)$ for $\Alg_{\E_n^\nu}(\C)$ for the category of non-unital $\E_n$-algebras. 

For an augmented algebra $A$, we write $\m(A) = \fib(A \to \one)$ for its augmentation ideal. A priori, it is simply an object of $\C$ but it canonically upgrades to a non-unital $\E_n$-algebra. In fact, when $\C$ is stably $\E_n$-monoidal, $\m\colon \Alg_{\E_n}^\aug(\C) \to \C$ canonically enhances to an equivalence of $\infty$-categories \cite[Prop.\ 5.4.4.10]{HA}
\begin{equation*}
\Alg_{\E_n}^{\mathrm{aug}}(\C) \xrightarrow{\,\simeq\,} \Alg_{\E_n}^{\mathrm{nu}}(\C)
\end{equation*}
whose inverse we denote by $X \mapsto \one \oplus X$.

When $\C$ is stably $\E_n$-monoidal, there is a canonical equivalence $\C \simeq \Sp(\Alg_{\E_n}^\aug(\C))$, compatible with the restriction maps $\Alg_{\E_n}^\aug(\C) \to \Alg_{\E_k}^\aug(\C)$ whenever $k\leq n$ as well as  compatible with lax $\E_n$-monoidal and exact functors.\footnote{Informally, given such a functor $f\colon \C \to \D$, it sends an augmented $\E_n$-algebra $A \to \one_\C$ in $\C$ to the augmented $\E_n$-algebra $F(A) \times_{F(\one_\C)} \one_\D$.} This follows from the canonical equivalence $\Mod_\one^{\E_n}(\C) \simeq \C$ \cite[Remark 7.3.5.3]{HA} and \cite[Theorem 7.3.4.13]{HA}. The composite of this equivalence followed by the functor $\Omega^\infty\colon \Sp(\Alg_{\E_n}^\aug(\C)) \to \Alg_{\E_n}^\aug(\C)$ is the trivial $\E_n$-algebra functor 
\[ \tr_{\E_n} \colon \C \to \Alg_{\E_n}^\aug(\C), \quad X \mapsto \one \oplus X \]
which exhibits $\tr_{\E_n}(X)$ is the trivial square zero extension on $X$. Informally speaking, it equips $X$ with the zero multiplication. Note that the functor $\tr_{\E_0}$ is an equivalence, as $\Alg_{\E_0}^\aug(\C) \simeq \C_{\one /\one} \simeq \C$ is already stable. In what follows, we will also view $\tr_{\E_n}$ as a functor on $\Alg^\aug_{\E_0}(\C)$ via this equivalence.
If $\C$ is presentably stably $\E_n$-monoidal, then $\tr_{\E_n}$ admits a left adjoint $\Sigma^\infty_{\E_n}\colon \Alg_{\E_n}^\aug(\C) \to \C$, often referred to as the $\E_n$-homology of an augmented $\E_n$-algebra.

\begin{remark}\label{rem:trivial-functors-compatible}
The trivial functors are compatible with forgetful functors. More precisely, when $\C$ is $\E_n$-monoidal, it is also $\E_{n-1}$-monoidal and the functor $\tr_{\E_{n-1}}$ is equivalent to the composite 
\[ \C \xrightarrow{\tr_{\E_n}} \Alg^\aug_{\E_n} \xrightarrow{\res_{\E_{n-1}}^{\E_n}} \Alg_{\E_{n-1}}^\aug(\C),\]
as follows from the compatibility of the equivalences $\C \simeq \Sp(\Alg_{\E_n}^\aug(\C))$ with the restriction maps as indicated above.

In particular, if $\C$ is symmetric monoidal, then $\tr_{\E_n}$ is equivalent to the composite
\[ \C \xrightarrow{\tr_{\E_\infty}} \Alg^\aug_{\E_\infty}(\C) \xrightarrow{\res_{\E_n}^{\E_\infty}} \Alg^\aug_{\E_n}(\C).\]
\end{remark}

For future reference, we record the following basic property of trivial algebras.
\begin{lemma}\label{lemma:trivial-and-base-change}
Let $f\colon \C \to \D$ an exact and lax-$\E_n$-monoidal functor between stably $\E_n$-monoidal categories. Then the induced functor $f\colon \Alg_{\E_n}^\nu(\C) \to \Alg_{\E_n}^\nu(\D)$ preserves trivial algebras. 
\end{lemma}
\begin{proof}
By assumption, $f\colon \C \to \D$ preserves finite limits. It follows that the induced functor $f \colon \Alg_{\E_n}^\aug(\C) \to \Alg_{\E_n}^\aug(\D)$ also preserves finite limits. This implies that the following diagram commutes.
\[\begin{tikzcd}
	\C \ar[r,"\simeq"] \ar[d,"f"'] & \Sp(\Alg_{\E_n}^\aug(\C)) \ar[r,"\Omega^\infty"] \ar[d,"\Sp(f)"'] & \Alg_{\E_n}^\aug(\C) \ar[d,"f"] \\
	\D \ar[r,"\simeq"] & \Sp(\Alg_{\E_n}^\aug(\D)) \ar[r,"\Omega^\infty"] & \Alg_{\E_n}^\aug(\D)
\end{tikzcd}\]
The claim then follows from the definition of the trivial functors as the horizontal composites.
\end{proof}

The functors $\tr_{\E_n}$ and $\omega^n$ as studied above are closely related:
\begin{theorem}\label{thm:trivial-algebras}
Let $\C$ be a stably $\E_n$-monoidal $\infty$-category. Then there is a canonical equivalence of functors $\omega^n \simeq \Omega^n\circ \tr_{\E_n} \colon \Alg_{\E_0}^\aug(\C) \to \Alg_{\E_n}^\aug(\C)$.
\end{theorem}
\begin{proof}
Again, we note that $\Alg^\aug_{\E_0}(\C)$ is canonically equivalent to $\C$ via the augmentation ideal. Then we note that the composite
\[ \Alg_{\E_0}^\aug(\C) \xrightarrow{\tr_{\E_n}} \Alg_{\E_n}^\aug(\C) \xrightarrow{\res^{\E_n}_{\E_0}} \Alg_{\E_0}^\aug(\C) \]
is canonically equivalent to the identity. In particular, we have
$\omega^n \simeq \omega^n \circ \res^{\E_n}_{\E_0} \circ \tr_{\E_n} \simeq \Omega^n \tr_{\E_n}$
by \cref{cor:loopsE_n}.
\end{proof}

Combined again with \cref{cor:loopsE_n} we obtain:
\begin{corollary}
\label{prop:nfoldsusploops}
Let $\C$ be a stably $\E_n$-monoidal $\infty$-category. For an augmented $\E_n$-algebra $A$ in $\C$ there is a natural equivalence
\begin{equation*}
\Omega^n(A) \simeq \tr_{\E_n}(\Omega^n(\res^{\E_n}_{\E_0}(A))).
\end{equation*}
\end{corollary}

By passing to left adjoints, we obtain the following corollary.
\begin{corollary}\label{remark:suspension-revisited}
Let $\C$ be presentably stably $\E_n$-monoidal. For an augmented $\E_n$-algebra $A$ in $\C$ there is a natural equivalence
\[ \Sigma^n_{\E_n}(A) \simeq \free_{\E_0}^{\E_n}(\Sigma^n(\Sigma^\infty_{\E_n}A)).\]
\end{corollary}

\begin{remark}
By passing to left adjoints in \cref{thm:trivial-algebras} we obtain an equivalence $\Barc^n(A) \simeq \Sigma^n \Sigma^\infty_{\E_n}A$ when $\C$ is presentably $\E_n$-monoidal.
This relation between the $n$-fold Bar construction and $\E_n$-homology is well-known, see e.g.\ \cite[Theorem 1.3]{BM} or \cite[Theorem 13.7]{GKRW}, but often proved via an explicit calculation, in contrast to the rather formal argument obtained here.
\end{remark}

\begin{remark}
Corollaries \ref{remark:suspension-revisited} and \ref{prop:nfoldsusploops} exhibit a kind of `fast stabilization' in the $\infty$-category of $\E_n$-algebras: any $n$-fold loop object is already an infinite loop object and  dually, an $n$-fold suspension is already an infinite suspension object. To put this into perspective, we note that for a stable category $\C$ and a non-unital operad $\O$, the stabilization of the $\infty$-category $\Alg_\O(\C)$ is always equivalent to $\C$ see again \cite[Theorem 7.3.4.13]{HA}. In this situation, it follows formally that any compact object in $\Alg_\O(\C)$, i.e., any finitely presented $\O$-algebra is in the essential image of the functor $\free_{\O}$ after finitely many suspensions. \cref{cor:suspensionEn} shows that for $\O = \E_n^{\mathrm{nu}}$ one has a much sharper statement; this happens after only $n$ suspensions, as well as for \emph{any} non-unital $\E_n$-algebra rather than just finitely presented ones.
\end{remark}

Elaborating on \cref{remark:suspension-revisited}, we obtain the following consequence, where we again implicitly identify $\C$ with $\Alg_{\E_0}^\aug(\C)$. This was already used in the PhD thesis of Shi \cite{Shi}.

\begin{corollary}
Let $\C$ be presentably stably $\E_n$-monoidal. Then the costabilization $\Sp(\Alg_{\E_n}^\aug(\C)^\op)^\op$ of $\Alg_{\E_n}^{\aug}(\C)$ is equivalent to $\C$. In particular, for any stable $\infty$-category $\D$ the functor
\begin{equation*}
\Fun^{\mathrm{rex}}(\D, \C)\lto \Fun^{\mathrm{rex}}(\D, \Alg_{\E_n}^{\aug}(\C)),\quad F \mapsto \free^{\E_n} \circ F
\end{equation*}
is an equivalence, where the superscript rex refers to functors which preserve finite colimits.
\end{corollary}
\begin{proof}
By \cite[Prop.\ 1.4.2.24]{HA} the costabilization of $\Alg_{\E_n}^{\aug}(\C)$ may be computed as the inverse limit of the system
\[
\begin{tikzcd}
\cdots \ar{r}{\Sigma_{\E_n}} & \Alg_{\E_n}^{\aug}(\C) \ar{r}{\Sigma_{\E_n}} & \Alg_{\E_n}^{\aug}(\C).
\end{tikzcd}
\]
Cofinality and \cref{remark:suspension-revisited} implies that this agrees with the inverse limit of the diagram
\[
\begin{tikzcd}
\cdots \ar{rr}{\Sigma^n_{\E_n}} \ar{dr}[swap]{\Sigma^n\Sigma^\infty_{\E_n}} && \Alg_{\E_n}^{\mathrm{nu}}(\C) \ar{rr}{\Sigma^n_{\E_n}} \ar{dr}[swap]{\Sigma^n\Sigma^\infty_{\E_n}} && \Alg_{\E_n}^{\mathrm{nu}}(\C) \\
& \C\ar{ur}{\free^{\E_n}} \ar[rr, dashed, ]&& \C\ar{ur}{\free^{\E_n}} &
\end{tikzcd}
\]
where the dashed composite is given by $\Sigma^n$ since the composite $\Sigma^\infty_{\E_n} \free^{\E_n}$ is canonically equivalent to the identity, as its right adjoint canonically identifies with the identity.
Applying cofinality again, the costabilization of $\Alg_{\E_n}^\aug(\C)$ agrees with the inverse limit of the diagram
\[
\begin{tikzcd}
\cdots \ar{r}{\Sigma^n} & \C \ar{r}{\Sigma^n} & \C.
\end{tikzcd}
\]
All the functors in this diagram are equivalences since $\C$ is stable and the corollary follows.
\end{proof}

\section{The stable $\E_n$-operads}
\label{sec:stableEn}

The aim of this section is to explain how the results of \cref{sec:susploopEn} can be reinterpreted as statements about the $\E_n$-operads themselves, at least when viewed as operads in spectra. In particular, we will see that in the special case of stable $\infty$-categories $\C$, Theorem \ref{thm:suspensionEn} is equivalent to the existence of a certain diagram of operad maps as follows:
\[
\begin{tikzcd}
\Sigma^\infty_+ \E_n^{\mathrm{nu}} \ar{r}\ar{d}\ar{dr}{\sigma} & \Sigma^\infty_+ \E_{n+1}^{\mathrm{nu}} \ar{d} \\
S \Sigma^\infty_+ \E_{n-1}^{\mathrm{nu}} \ar{r} & S \Sigma^\infty_+ \E_n^{\mathrm{nu}}.
\end{tikzcd}
\]
Here $\Sigma^\infty_+\E_n^{\mathrm{nu}}$ is the operad in the $\infty$-category of spectra obtained by taking the levelwise suspension spectrum of the nonunital operad $\E_n^\nu$ in spaces. Also, $S\Sigma^\infty_+\E_n^{\mathrm{nu}}$ denotes the \emph{operadic suspension} of $\Sigma^\infty_+\E_n^\nu$, which we discuss in Section \ref{subsec:prelim}, and $\sigma$ denotes the \emph{suspension morphism}. For the corresponding free algebra functors, this morphism can be described as the canonical suspension map $\sigma\colon \free_{\E_n}(X) \to \Sigma^{-1}\free_{\E_n}(\Sigma X)$. The horizontal morphisms in the diagram are the standard inclusions. The vertical arrows are a bit more exotic; they can be thought of as the `Koszul dual' maps to the standard inclusions, cf. Remark \ref{rmk:wrongway}. 

The main technical tool in this section is the result that there is no loss of information in passing from the stable $\E_n$-operads to their associated monads. More precisely, we show in Theorem \ref{thm:SymffEuler} that there is a class of operads $\mathcal{O}$ `with nilpotent Euler classes' for which the assignment
\[
\Operad(\Sp) \lto \Monad(\Sp),\quad \mathcal{O} \mapsto \Sym_{\mathcal{O}}
\]
is fully faithful. In Section \ref{subsec:operadsEulerclasses} we demonstrate that the stable $\E_n$-operads and their operadic (de)suspensions satisfy the conditions of that theorem.

It should be true that the square of operad maps above already exists \emph{unstably}, i.e.\ in the setting of operads in pointed spaces without applying $\Sigma^\infty$ anywhere. To work this out one would need a good handle on the operadic suspension and a possible universal property in the unstable case. We comment on existing work and what would need to be done to obtain this sharpening in Section \ref{subsec:unstableEn}.

\subsection{Preliminaries}
\label{subsec:prelim}

We denote by $\SSeq(\Sp) := \Fun(\Fin^\simeq,\Sp)$ the category of symmetric sequences in spectra. Here $\Fin^\simeq$ denotes the groupoid of finite sets and bijections, which is the free symmetric monoidal category on a single generator. We will usually write $A(n)$ for the value of a symmetric sequence $A$ on the set $\{1, \ldots, n\}$. 
A symmetric sequence $A$ is called non-unital if $A(0)$ is equivalent to $0$. The category $\SSeq(\Sp)$ is equipped with the Day convolution symmetric monoidal structure $\otimes$ based on the usual symmetric monoidal structure on $\Sp$ and the disjoint union symmetric monoidal structure on $\Fin^\simeq$. By construction, the fully faithful functor $\iota\colon \Sp \to \SSeq(\Sp)$, sending $X$ to the symmetric sequence $\iota(X)(0) = X$ and $\iota(X)(n) = 0$ for $n>0$ is canonically symmetric monoidal. In particular, $\SSeq(\Sp)$ is tensored over $\Sp$ via this inclusion.

The $\infty$-category $\SSeq(\Sp)$ can be equipped with a further monoidal structure, called the \emph{composition product}, obeying the formula 
\[ A \circ B = \bigoplus\limits_{n \geq 0} (A(n) \otimes B^{\otimes n})_{h\Sigma_n},\]
see \cite[Section 3.1]{Brantner} and \cite{Haugseng}.
It follows that $A \circ \iota(X)$ is in the essential image of $\iota$ for any symmetric sequence $A$, giving rise to an action of $\SSeq(\Sp)$ on $\Sp$, or equivalently, a monoidal functor 
\[ \free\colon \SSeq(\Sp) \to \End(\Sp), \quad A \mapsto [\free_A \colon X \mapsto \bigoplus_{n \geq 0} (A(n)\otimes X^{\otimes n})_{h\Sigma_n}]\]
where we view $\End(\Sp) = \Fun(\Sp,\Sp)$ as monoidal via the composition of functors.
It will be convenient to use the abbreviated notation $D_n^A(X) := (A(n)\otimes X^{\otimes n})_{h\Sigma_n}$ for the summands of $\free_A$.

For non-unital symmetric sequences $A$ and $B$, the above formula for the composition product simplifies to the following. For a finite set $I$ we have
\[
(A \circ B)(I) \simeq \bigoplus_{E \in \mathrm{Equiv}(I)} A(I/E) \otimes \bigotimes_{J \in I/E} B(J),
\]
where the sum is indexed over the set of equivalence relations $E$ on $I$. It will be important for us to make explicit the $\Sigma_I$-action on $(A\circ B)(I)$. To that end, first note that $\Sigma_I$ acts on the set $\Equiv(I)$ of equivalence relations on $I$ in the obvious fashion. For $E \in \Equiv(I)$ let $G_E$ denote the stabilizer of the action so that there are canonical induced maps $G_E \to \Sigma_{I/E}$ as well as $G_E \to \Sigma_J$ for all $J \in I/E$.
Then we have an equivalence of $\Sigma_I$-objects
\[ (A\circ B)(I) \simeq \bigoplus\limits_{E \in \Equiv(I)/\Sigma_I} \ind_{G_E}^{\Sigma_I} [A(I/E) \otimes \bigotimes_{J \in I/E} B(J)]\]
for the diagonal $G_E$ action on $A(I/E) \otimes \bigotimes_J B(J)$.

An \emph{operad} in $\Sp$ is defined to be an algebra object in the monoidal $\infty$-category $\SSeq(\Sp)$. We write $\Operad(\Sp) = \Alg(\SSeq(\Sp))$ for the $\infty$-category of such. An operad is said to be non-unital when its underlying symmetric sequence is non-unital. Since algebras in $\End(\Sp)$ are monads and $\free$ is monoidal, we obtain a functor
\[
\free\colon \Operad(\Sp) \lto \Monad(\Sp),\quad \O \mapsto \Sym_{\O}.
\]
We note that the monad $\free_\O$ associated to a non-unital operad $\O$ is a reduced functor.

\subsubsection*{The operadic suspension}
We write $\bS^{1}[1]$ for the symmetric sequence that assigns the spectrum $\bS^1$ to a singleton and 0 to a set of any cardinality other than 1. Observe that $\Sym_{\bS^{1}[1]}$ is the suspension functor $\Sigma \in \End(\Sp)$. As follows from the above formulas for the composition product, the object $\bS^{1}[1]$ has a monoidal inverse, namely $\bS^{-1}[1]$ with $\free_{\bS^{-1}[1]} \simeq \Sigma^{-1} \in \End(\Sp)$. Conjugating by the invertible object $\bS^{1}[1]$ therefore defines a monoidal automorphism \cite[Lemma 3.10]{Brantner}
\[
S\colon \SSeq(\Sp) \lto \SSeq(\Sp),\quad A \mapsto \bS^{-1}[1] \circ A \circ \bS^{1}[1].
\]
\begin{definition}
Since $S$ is monoidal it induces an equivalence $S\colon \Operad(\Sp) \to \Operad(\Sp)$. For an operad $\O$, we call $S\O$ its \emph{operadic suspension} and $S^{-1}\O$ its \emph{operadic desuspension}.
\end{definition}

\begin{remark}
Let us explicitly describe the underlying symmetric sequence of the operadic suspension of a non-unital operad $\O$. Consider the standard $n$-dimensional real representation $\R^n$ of the symmetric group $\Sigma_n$ given by permuting the standard basis vectors of $\R^n$. The diagonal $\Delta$ is a trivial subrepresentation; we write $\rho_n = \R^n/\Delta$ for the quotient. As usual, we write $S^V$ for the one-point compactification of a representation $V$ and $\bS^V$ for the suspension spectrum of $S^V$. There are then equivalences of representation spheres
\[
\bS^{\R^n} \simeq (\bS^1)^{\otimes n} \quad\quad \text{and} \quad\quad \bS^{\rho_n} \simeq \Sigma^{-1}(\bS^1)^{\otimes n},
\]
where the symmetric group $\Sigma_n$ acts by permuting the factors on the right-hand sides. As a consequence, we find equivalences $S\O(n) \simeq \bS^{\rho_n} \otimes \O(n)$. Applying this for $\O = \Sigma^\infty_+ \E_\infty^\nu$, we find that $\bS^\rho$, the symmetric sequence sending $n$ to $\bS^{\rho_n}$, is a non-unital operad and that $S\O$ is obtained from the pointwise tensor product of the operads $\bS^\rho$ and $\O$.

Clearly, the operad $\bS^\rho$ ought to be in the image of the functor $\Operad(\Spc_*) \to \Operad(\Sp)$. Indeed, Arone--Kankaanrinta \cite{aronekankaanrinta} and Ching--Salvatore \cite{chingsalvatore} both give explicit constructions of a \emph{sphere operad} in the category of pointed spaces that should provide such a desuspension of $\bS^\rho$. It would be desirable to have a comparison between their work and the discussion we gave above, but we do not pursue the matter here. We will discuss the unstable case further in Section \ref{subsec:unstableEn}.
\end{remark}

By construction, the monad associated to the operad $S\O$ is described by 
\[
\Sym_{S\mathcal{O}} \simeq \Sigma^{-1}\circ \Sym_{\O}\circ \Sigma.
\]
Moreover, since $\Alg_\O(\Sp) = \Alg_{\free_\O}(\Sp)$, there is a commutative square of $\infty$-categories
\[
\begin{tikzcd}
\Alg_{\mathcal{O}}(\Sp) \ar{d}{\mathrm{frgt}}\ar{r} & \Alg_{S\mathcal{O}}(\Sp) \ar{d}{\mathrm{frgt}} \\
\Sp \ar{r}{\Sigma} & \Sp
\end{tikzcd}
\]
obtained by letting the automorphism $\Sigma \colon \Sp \to \Sp$ act on $\Monad(\Sp)$ via conjugation.
Both horizontal functors in this diagram are equivalences, so the category of $S\O$-algebras is equivalent to that of $\O$-algebras via the suspension functor on the level of underlying objects. 

To conclude this section, we discuss a \emph{suspension morphism} $\sigma \colon \O \to S\O$ of a non-unital operad $\O$.
Let us first describe the effect $\sigma$ on the free algebra monads of $\O$ and $S\O$, respectively. To that end, let $T$ be a monad on $\Sp$ that is reduced, i.e. $T(0) \cong 0$. Then there are commuting squares
\[
\begin{tikzcd}
\Alg_T(\Sp) \ar{r}{\Sigma_T} & \Alg_T(\Sp) && \Alg_T(\Sp) \ar{d}{\fgt_T} & \Alg_T(\Sp) \ar{d}{\fgt_T}\ar{l}[swap]{\Omega_T} \\
\Sp \ar{u}{\Sym_T}\ar{r}{\Sigma} & \Sp \ar{u}{\Sym_T} && \Sp & \Sp \ar{l}{\Omega}
\end{tikzcd}
\]
of left and right adjoints, respectively, with $\Sigma_T$ and $\Omega_T$ denoting the suspension and loops functor internal to the pointed $\infty$-category $\Alg_T(\Sp)$. The unit of that adjunction supplies a map of monads on $\Sp$ of the form
\[
T \to \fgt_T \Omega_T \Sigma_T \free_T \simeq \Omega \fgt_T\free_T\Sigma = \Omega T \Sigma
\]
where the equivalence follows from the commutativity of the squares. In the specific case $T = \Sym_{\O}$ for a non-unital operad $\O$, we find a map of monads
\[
\sigma\colon \Sym_{\O} \lto \Omega \Sym_{\O} \Sigma \simeq \Sym_{S\O}.
\]
We claim that this map arises by applying the free functor to a map of operads $\O \to S\O$ for which we also write $\sigma$. We will not need a general construction here, but we will provide it in the specific case $\O = \E_n^{\mathrm{nu}}$ in \cref{subsec:suspensionE_n} below, see also \cref{Remark:goodwillie-calculus-differentiate-monads}. In particular, the map $\sigma$ yields a commutative diagram as follows:
\[
\begin{tikzcd}
\Alg_{\O}(\Sp) \ar[d,"\fgt_\O"] & \Alg_{S\O}(\Sp) \ar{l}[swap]{\sigma^*} \ar[d,"\fgt_{S\O}"] & \Alg_{\O}(\Sp) \ar{l}[swap]{\simeq}\ar[d,"\fgt_\O"] \\
\Sp & \Sp \ar[equals]{l} & \Sp \ar{l}{\Omega}
\end{tikzcd}
\]
Here the equivalence $\Alg_{\O}(\Sp) \simeq \Alg_{S\O}(\Sp)$ in the top right is the inverse of the equivalence described above. Thus, we may interpret the loop functor on $\Alg_{\O}(\Sp)$ as restriction along the morphism $\sigma$, up to an equivalence of categories.

\begin{remark}\label{Remark:goodwillie-calculus-differentiate-monads}
We recall that the formation of Goodwillie derivatives forms a functor
\[
\partial_*\colon \End(\Sp) \to \SSeq(\Sp)
\]
from endofunctors of $\Sp$ to symmetric sequences in $\Sp$. If this construction would be (lax) monoidal, then it would induce a functor from monads in $\Sp$ to operads in $\Sp$. This monoidality is essentially the statement of the chain rule for the derivatives of endofunctors of $\Sp$, but it has not quite been established in this form in the literature. A version of the chain rule that does accomplish this is the subject of ongoing work of Blans and Blom \cite{blansblom}. Moreover, in loc.\ cit.\ a monoidal equivalence $\partial_* \circ \free \simeq \id_{\SSeq(\Sp)}$ is established. One could therefore \emph{define} the suspended operad $S\O$ as $\partial_*(\Omega \free_\O \Sigma)$ and the suspension morphism as $\partial_*(\sigma \colon \free_\O \to \Omega \free_\O \Sigma)$. We will show later that this construction is equivalent to our construction of the map $\sigma\colon \O \to S\O$ in case $\O = \Sigma^\infty_+ \E_n^\nu$; note that we do not give a construction of $\sigma$ otherwise.
\end{remark}

\subsection{Symmetric sequences with nilpotent Euler classes}
\label{subsec:operadsEulerclasses}

The functor
\begin{equation*}
\SSeq(\Sp) \lto \Fun(\Sp,\Sp), \quad A \mapsto \Sym_A
\end{equation*}
is well known not to be fully faithful. Take for instance $A = \E_\infty$. Then there is a non-trivial map $\id = D^{\E_\infty}_1 \to D^{\E_\infty}_2$, as the space of such natural transformations is equivalent to $\Omega^\infty(\bS^{-1})^{tC_2} \simeq \Omega^{\infty+1}\bS_2$,\footnote{See the discussion around \eqref{dual-derivative} for this argument.} which contains the non-trivial element $\eta$. This induces a non-trivial self-map of $\free_{\E_\infty}$ which does not come from a self-map of the symmetric sequence $\E_\infty$. However, we will show in \cref{thm:SymffEuler} that the functor $A \mapsto \Sym_A$ \emph{does} become fully faithful after restricting its domain to a certain full subcategory of symmetric sequences, the ones with \emph{nilpotent Euler classes}, to which this section is devoted.

Let $T$ be a finite set and let us write $\rho_T = \R[T]/\Delta$ for the reduced standard representation of the symmetric group $\Sigma_T$. The inclusion of the origin into $\rho_T$ induces a $\Sigma_T$-equivariant map
\begin{equation*}
e_T\colon S^0 \to S^{\rho_T}
\end{equation*}
which we refer to as the \emph{Euler class}. 

\begin{definition}
\label{def:nilpotentEuler}
A symmetric sequence $A \in \SSeq(\Sp)$ has \emph{nilpotent Euler classes} if there exists a natural number $k$, called the order of nilpotence, with the following property: for each $n \geq 2$ and each subset $T \subseteq \{1, \ldots, n\}$ of cardinality at least two, the map
\begin{equation*}
A(n) \simeq A(n) \otimes S^0 \xrightarrow{\id \otimes e_T^k} A(n) \otimes S^{k\rho_T}
\end{equation*}
is nullhomotopic in the $\infty$-category $\Fun(B\Sigma_T,\Sp)$, where $A(n)$ is interpreted as an object of $\Fun(B\Sigma_T,\Sp)$ via restriction along the evident inclusion $\Sigma_T \subseteq \Sigma_n$. Note that $k$ should not depend on $n$ or $T$. An operad $\O \in \Operad(\Sp)$ is said to have nilpotent Euler classes if its underlying symmetric sequence does.
\end{definition}

\begin{example}
\begin{enumerate}
\item Suppose that $A$ is a symmetric sequence such that for every $n\geq 2$ the spectrum $A(n)$ is a $\Sigma_n$-free object, e.g.\ $A = \Sigma^\infty_+\E_1^\nu$. Then $A$ has nilpotent Euler classes with order of nilpotence equal to $1$. Indeed, for any $T \subseteq \{1,\dots,n\}$, the $\Sigma_T$ object $A(n)$ is a sum of free $\Sigma_T$-objects, hence the claim follows from the fact that the Euler class $e_T \colon S^0 \to S^{\rho_T}$ is nonequivariantly null when $|T|\geq 2$.
\item The symmetric sequence underlying the operad $\Sigma^\infty_+\E_\infty^\nu$ does not have nilpotent Euler classes. For example, taking homotopy orbits of the $\Sigma_2$-equivariant map
$
e_2^k\colon \bS^0 \to \bS^{k\rho_2}
$
yields the standard map
$
\Sigma^\infty \RP^\infty_+ \to \Sigma^\infty \RP^\infty_k,
$
which is not nullhomotopic for any (finite) value of $k$.
\end{enumerate}

\end{example}

\begin{lemma}
\label{lem:closureSSeqEuler}
The class of non-unital symmetric sequences with nilpotent Euler classes is closed under operadic (de)suspensions and contains the monoidal unit. Moreover, if $A$ is a non-unital symmetric sequence with nilpotent Euler classes, then so is its composition product $A \circ B$ with any non-unital symmetric sequence $B$. In particular, the class of non-unital symmetric sequences with nilpotent Euler classes forms a monoidal subcategory of $\SSeq(\Sp)$.
\end{lemma}
\begin{proof}
The monoidal unit of $\SSeq(\Sp)$ is the sequence that associates $\bS$ to the singleton and $0$ to any finite set of cardinality other than one, and consequently has nilpotent Euler classes.

Now let $A$ be a symmetric sequence. Observe that the Euler classes
\[
(SA)(n) \xrightarrow{e^k_T} (SA)(n) \otimes S^{k\rho_T}
\]
for $SA$ are obtained from those of $A$ by smashing with the invertible object $\bS^{\rho_n}$ of $\Fun(B\Sigma_n,\Sp)$. Hence the operadic suspension $SA$ of $A$ has nilpotent Euler classes if and only if $A$ itself does.

Finally, consider a composition $A \circ B$ where $A$ has nilpotent Euler classes, say with order of nilpotence $k$. Let $T$ be a subset of $I := \{1, \ldots, n\}$. As in our discussion of the composition product above, the restriction of the $\Sigma_I$-object
\[
(A \circ B)(I) \cong \bigoplus_{E \in \mathrm{Equiv}(I)} A(I/E) \otimes \bigotimes_{J \in I/E} B(J)
\]
to $\Sigma_T \subseteq \Sigma_n$ splits into summands corresponding to the different orbits of $\mathrm{Equiv}(I)/\Sigma_T$. Consider an $E \in \mathrm{Equiv}(I)$ with stabilizer $G_E^T \leq \Sigma_T$. Then it will suffice to show that
\[
A(I/E) \otimes \bigotimes_{J \in I/E} B(J) \xrightarrow{e^k_T} A(I/E) \otimes \bigotimes_{J \in I/E} B(J) \otimes S^{k\rho_T}
\]
is $G_E^T$-equivariantly null, where we recall $k$ to be the order of nilpotence of $A$. If $G_E^T$ is a non-transitive subgroup of $\Sigma_T$, then this is immediate from the fact that the restriction of the Euler class $e_T$ itself to $G_E^T$ is null: indeed, then the representation $\rho_T$ admits a nonzero $G_E^T$-fixed point. If $G_E^T$ is a transitive subgroup of $\Sigma_T$, then the restriction of the equivalence relation $E$ to $T$ must be a partition into blocks of equal size; let us say there are $a$ blocks of size $b$. Then $G_E^T \cong \Sigma_b \wr \Sigma_a$ and $T/E \cong \{1, \ldots, a\}$. Write $q\colon G_E^T \to \Sigma_a$ for the quotient by the normal subgroup $\Sigma_b^a$. 

The standard representation $\R^T$ of $\Sigma_T$, when restricted to $\Sigma_b \wr \Sigma_a$, can be written as $\R^a \otimes \R^b$ in an evident way. As a $\Sigma_b$-representation, we can split $\R^b$ as $\Delta \oplus \rho_b$, simply by averaging. Hence we find an isomorphism of representations
\[
\R^a \otimes \R^b \cong \R^a \oplus (\R^a \otimes \rho_b).
\]
From this we obtain isomorphisms of $\Sigma_b \wr \Sigma_a$-objects
\[
\rho_T \cong \rho_a \oplus \bigoplus_{i=1}^a \rho_b \quad\quad \text{and hence} \quad\quad S^{\rho_T} \cong S^{\rho_a} \otimes (S^{\rho_b})^{\otimes a}.
\]
We now observe that the map $e_k^T$ under consideration admits a tensor factor
\[
A(I/E) \to A(I/E) \otimes S^{k\rho_a}
\]
on which the subgroup $\Sigma_b^a$ acts trivially. More precisely, this factor is in the image of $q^*$. Hence it suffices to show that this map is $\Sigma_a$-equivariantly null. That this is so follows from the assumption that $A$ has nilpotent Euler classes of order $k$, after identifying $T/E$ with its image in $I/E$ and $\Sigma_a$ with the corresponding subgroup of the permutation group of $I/E$.
\end{proof}

Our main source of operads with nilpotent Euler classes is the following:

\begin{proposition}
\label{prop:EnoperadEuler}
For each $m \geq 1$, the operad $\Sigma^\infty_+ \E_m^\nu$ has nilpotent Euler classes.
\end{proposition}
\begin{proof}
We will show that $k=m$ satisfies the condition of Definition \ref{def:nilpotentEuler}. Recall that $\E_m^\nu(n)$ is (homotopy equivalent to) the configuration space $\mathrm{Conf}_n(\R^m)$ of $n$ points in $\R^m$. The map
\[
\Sigma^\infty_+\E_m^\nu(n) \otimes S^0 \xrightarrow{e_T^m} \Sigma^\infty_+\E_m^\nu(n) \otimes S^{m\rho_T}
\]
is the infinite suspension of the map of pointed $\Sigma_T$-spaces
\[
\mathrm{Conf}_n(\R^m)_+ \xrightarrow{e_T^m} \mathrm{Conf}_n(\R^m)_+ \wedge S^{m\rho_T}.
\]
The right-hand side may be interpreted as the Thom space of the $\Sigma_T$-equivariant (trivial) vector bundle $\mathrm{Conf}_n(\R^m) \times \rho_T^{\oplus m}$ over the configuration space. The map $e_T^m$ then includes $\mathrm{Conf}_n(\R^m)$ as the zero-section of this bundle and sends the disjoint basepoint $+$ to the basepoint `at infinity'. To show that the map $e_T^m$ is equivariantly nullhomotopic, it suffices to construct a nowhere vanishing equivariant section of this bundle or, in other words, a nowhere vanishing equivariant function
\[
f = (f_1, \ldots, f_m)\colon \mathrm{Conf}_n(\R^m) \lto \rho_T^{\oplus m}.
\]
Indeed, multiplying such a section by a scalar $t$ and letting $t$ range from 0 to $\infty$ then defines a homotopy from the zero-section to the constant map sending all of $\mathrm{Conf}_n(\R^m)$ to $\infty$.

Let $\mathrm{Conf}_T(\R^m)$ denote the space of configurations of the set $T$ inside $\R^m$. Then it suffices to construct a nowhere vanishing equivariant map 
\[
f = (f_1, \ldots, f_m)\colon \mathrm{Conf}_T(\R^m) \lto \rho_T^{\oplus m},
\]
since the previous case follows after precomposing with the forgetful map $\mathrm{Conf}_n(\R^m) \to \mathrm{Conf}_T(\R^m)$. For $1 \leq i \leq m$, let 
\[
g_i\colon \mathrm{Conf}_T(\R^m) \lto \R^T
\]
be the map taking the $i$th coordinate of every point in a given configuration. Set $f_i$ to be the composition of $g_i$ with the quotient map $\R^T \to \rho_T$. Since the $g_i$ cannot all take values in the diagonal simultaneously, the map $f$ has no zeros.
\end{proof}

\subsection{From operads to monads}

The goal of this section is to prove the following:
\begin{theorem}
\label{thm:SymffEuler}
The functors
\[
\Sym\colon \SSeq(\Sp) \lto \End(\Sp) \quad\quad \text{and} \quad\quad \Sym\colon \Operad(\Sp) \lto \Monad(\Sp)
\]
become fully faithful when restricted to the full subcategories of non-unital symmetric sequences (resp.\ non-unital operads) with nilpotent Euler classes.
\end{theorem}

Our argument is inspired by \cite[Appendix B]{heuts}, where the corresponding result is proved for symmetric sequences and operads in $T(n)$-local spectra (where the hypothesis on nilpotent Euler classes is unnecessary). As in loc.\ cit., we will use \emph{dual Goodwillie calculus} for endofunctors of $\mathrm{Sp}$, which is nothing but Goodwillie calculus applied to the opposite category $\Sp^{\mathrm{op}}$, cf.\ \cite[Appendix A]{heuts}. Goodwillie calculus itself supplies for any functor $F\colon \Sp \to \Sp$ a tower of approximations from the right
\begin{equation*}
F \to \cdots \to P_n F \to P_{n-1} F \to \cdots \to P_1 F \to P_0 F
\end{equation*}
with the property that $F \to P_n F$ is the initial map from $F$ to an $n$-excisive functor. Dual Goodwillie calculus then gives a cotower of approximations from the left
\[
P^0 F \to \cdots \to P^{n-1} F \to P^n F \to \cdots \to F
\]
with the property that $P^n F \to F$ is the terminal map to $F$ from an $n$-excisive functor.\footnote{Since $\Sp$ is stable, there is no distinction between $n$-excisive and $n$-coexcisive functors.}

The key to our proof of \cref{thm:SymffEuler} is the following observation:

\begin{lemma}
\label{lem:dualcalcEuler}
If the symmetric sequence $B$ has nilpotent Euler classes, then the evident inclusion
\[
\bigoplus_{k=1}^n D_k^B \lto  \bigoplus_{k=1}^\infty D_k^B = \Sym_B
\]
induces an equivalence
\[
\bigoplus_{k=1}^n D_k^B  \xrightarrow{\,\simeq\,} P^n\Sym_B.
\]
\end{lemma}
\begin{remark}
The functor $P_n$ commutes with filtered colimits, making it straightforward to calculate that $P_n\Sym_B$ is also given by the sum of the first $n$ homogeneous layers. However, the dual approximation $P^n$ generally does \emph{not} commute with filtered colimits. Nonetheless, the lemma shows that it commutes with the infinite direct sum $\bigoplus_{n \geq 1} D_n^B$ in case $B$ has nilpotent Euler classes.
\end{remark}

We will prove Lemma \ref{lem:dualcalcEuler} after we show how it implies Theorem \ref{thm:SymffEuler}:

\begin{proof}[Proof of Theorem \ref{thm:SymffEuler}]
By passing to algebra objects, \cref{lem:closureSSeqEuler} implies that it suffices to show that the functor $\free \colon \SSeq(\Sp) \to \End(\Sp)$ is fully faithful when restricted to the full subcategory of symmetric sequences with nilpotent Euler classes.
To that end, we consider the commutative diagram
\[\begin{tikzcd}
	\Map_{\SSeq(\Sp)}(A,B) \ar[r] \ar[d,"\simeq"'] & \Map_{\End(\Sp)}(\free_A,\free_B) \ar[d,"\simeq"] \\
	\prod\limits_{n\geq1} \Map_{\Sp^{B\Sigma_n}}(A(n),B(n)) \ar[dr, bend right=17,"\simeq"] & \prod\limits_{n\geq 1} \Map_{\End(\Sp)}(D_n^A,\free_B) \\
	& \prod\limits_{n\geq1} \Map_{\End(\Sp)}(D_n^A,D_n^B) \ar[u] 
\end{tikzcd}\]
in which we aim to show that the upper horizontal map is an equivalence. The left vertical map is an equivalence since $\Fin^\simeq$ is equivalent to the coproduct of the categories $B\Sigma_n$ and the right downward vertical map is an equivalence since $\free_A$ is the coproduct of the functors $D^A_n$. The diagonal arrow is an equivalence by Goodwillie's classification of homogenous functors. It hence remains to show that the lower right vertical arrow is also an equivalence. By definition, it is given by the product over $n\geq 1$ of the following composites.
\[ \Map(D_n^A,D_n^B) \lto \Map(D_n^A,\bigoplus_{k=1}^n D_k^B) \lto \Map(D_n^A,P^n\free_B) \lto \Map(D_n^A,\free_B)\]
The first map in this composite is an equivalence since for $k<n$, $D_k^B$ is $k$-excisive and $P_k(D^A_n) = 0$. The second map is an equivalence by \cref{lem:dualcalcEuler} and the third map is an equivalence since $D_n^A$ is $n$-excisive.
This completes the argument that $\Sym$ is fully faithful on symmetric sequences with nilpotent Euler classes and hence the proof of the theorem.
\end{proof}

\begin{proof}[Proof of \cref{lem:dualcalcEuler}]
It will suffice to show that $P^n\bigl(\bigoplus_{k>n} D_k^B\bigr) \cong 0$. We can reduce further to checking that the dual derivative $D^{\ell}\bigl(\bigoplus_{k>n} D_k^B\bigr)$ vanishes for each $\ell \leq n$. Generally, the dual derivative $D^{\ell} F$ of a functor $F \in \End(\Sp)$ may be constructed as follows. First one forms the \emph{cocross effect} of $F$, which is the functor of $\ell$ variables defined as the total cofibre
\[
\mathrm{cr}^{\ell} F(X_1, \ldots, X_{\ell}) := \mathrm{tcof}(F \circ \mathcal{X}),
\]
where $\mathcal{X}$ is the cube
\[
\mathcal{X}\colon \mathcal{P}(\ell) \rightarrow \Sp\colon T \mapsto \bigoplus_{t \in T} X_t.
\]
Here $\mathcal{P}(\ell)$ is the power set of $\{1, \ldots, \ell\}$ regarded as a poset under inclusion. Now one colinearizes $\mathrm{cr}^{\ell} F$ in each variable to obtain a multilinear functor $\partial^{\ell} F$ described by
\[
(\partial^{\ell} F)(X_1, \ldots, X_\ell) := \varprojlim_j \Sigma^{j\ell}(\mathrm{cr}^{\ell} F)(\Sigma^{-j} X_1, \ldots, \Sigma^{-j} X_\ell).
\]
Finally, the dual derivative is then
\[
D^{\ell} F(X) = \partial^{\ell} F(X, \ldots, X)^{h\Sigma_{\ell}}.
\]
We are analyzing the specific case $F = \bigoplus_{k>n} D_k^B$. To show that $D^{\ell} F$ vanishes, we will argue that $\partial^{\ell} F$ vanishes. The cocross effect $\mathrm{cr}^{\ell} F(X_1, \ldots, X_{\ell})$ is easily seen to be the direct sum over $k > n$ of the terms
\[
\Bigl(\bigoplus_{f\colon \mathbf{k} \to \mathbf{l}} B(k) \otimes \bigotimes_{i=1}^{\ell} X_i^{\otimes f^{-1}(i)}\Bigr)_{h\Sigma_k}.
\]
Here the sum is indexed over surjections $f\colon \{1, \ldots, k\} \to \{1, \ldots, \ell\}$. As a consequence, the inverse system $\Sigma^{j\ell}(\mathrm{cr}^{\ell} F)(\Sigma^{-j} X_1, \ldots, \Sigma^{-j} X_\ell)$ is the direct sum over $k > n$ of the homotopy $\Sigma_k$-orbits of the inverse systems in $j$
\[
\bigoplus_{f\colon \mathbf{k} \to \mathbf{l}} B(k) \otimes \bigotimes_{i=1}^{\ell} \bS^{-j\rho_{f^{-1}(i)}} \otimes \bigotimes_{i=1}^{\ell} X_i^{\otimes f^{-1}(i)}.
\]
Here the spectrum $\bigotimes_{i=1}^{\ell} \bS^{-j\rho_{f^{-1}(i)}}$ denotes a smash product of (virtual) representation spheres corresponding to the stabilizer $\Sigma_{f^{-1}(1)} \times \cdots \times \Sigma_{f^{-1}(\ell)}$ of a given surjection $f$. We conclude that the expression above is induced, as a spectrum with $\Sigma_k$-action, from a Young subgroup of the form $\Sigma_{k_1} \times \cdots \times \Sigma_{k_{\ell}}$, with $k_1 + \cdots + k_{\ell} = k$. Since $k > \ell$ there exists an $i$ with $k_i \geq 2$.

The conclusion of the lemma will follow if we argue that there exists a constant $C$, independent of $k$ and $f$, such that each of the maps 
\[
B(k) \otimes \bigotimes_{i=1}^{\ell} \bS^{-C\rho_{k_i}} \to  B(k)
\]
is null in $\Fun(B\Sigma_{k_1} \times \cdots \times B\Sigma_{k_{\ell}},\Sp)$. But this map admits a factor of the form
\[
B(k) \otimes \bS^{-C\rho_{k_i}} \to B(k)
\]
with $k_i \geq 2$, which is null if $C$ is at least the order of nilpotence of the Euler classes of $B$.
\end{proof}

\subsection{The suspension morphism of $\E_n$}
\label{subsec:suspensionE_n}

The main result of Section \ref{sec:susploopEn} was Theorem \ref{thm:suspensionEn}, providing two different factorizations of the suspension functor on the $\infty$-category of $\E_n$-algebras. We will now see how, in the stable context, this statement can also be considered as one about the $\E_n$-operad itself. In the following theorem we have abbreviated the notation $\Sigma^\infty_+\E_n^{\mathrm{nu}}$ to simply $\E_n$, and similarly for the other operads. 

\begin{theorem}
\label{thm:factorizationsuspension}
There is an essentially unique commutative square
\[
\begin{tikzcd}
\E_n \ar{r}{\iota_n}\ar{d}[swap]{\beta_n}\ar{dr}{\sigma} & \E_{n+1} \ar{d}{\beta_{n+1}} \\
S\E_{n-1} \ar{r}[swap]{S\iota_{n-1}} & S\E_n
\end{tikzcd}
\]
of non-unital operads in $\Sp$ such that the resulting diagram of $\infty$-categories of algebras
\[
\begin{tikzcd}
\Alg_{\E_n}^{\mathrm{nu}}(\Sp) \ar{r}\ar{d}\ar{dr} & \Alg_{\E_{n+1}}^{\mathrm{nu}}(\Sp) \ar{d} \\
\Alg_{S\E_{n-1}}^{\mathrm{nu}}(\Sp) \ar{r} & \Alg_{S\E_n}^{\mathrm{nu}}(\Sp)
\end{tikzcd}
\]
is equivalent to the diagram of \cref{thm:suspensionEn} in the special case $\C = \Sp$.
\end{theorem}
\begin{proof}
Adding the free functors for the various algebra categories to the square of \cref{thm:suspensionEn} and identifying non-unital with augmented algebras, gives the following commutative diagram of left adjoint functors.
\[
\begin{tikzcd}
	\Sp \ar[dr,"\free^{\E_n}"] \ar[drr, bend left= 15,"\free^{\E_{n+1}}"] \ar[ddr, bend right=15, "\free^{\E_{n-1}} \circ \Sigma"'] & & \\
	& \Alg_{\E_n}^{\mathrm{nu}}(\Sp) \ar[r,"\free_{\E_n}^{\E_{n+1}}"] \ar{d}{\Barc} \ar[dr, "\Sigma_{\E_n}"] & \Alg_{\E_{n+1}}^{\mathrm{nu}}(\Sp) \ar{d}{\Barc} \\
	& \Alg_{\E_{n-1}}^{\mathrm{nu}}(\Sp) \ar[r,"\free_{\E_{n-1}}^{\E_n}"'] & \Alg_{\E_n}^{\mathrm{nu}}(\Sp)
\end{tikzcd}\]
We should verify that the two triangles added to the square indeed commute. This is clear for the upper one, whereas for the left one it follows from the natural equivalences
\[
\Barc\circ \free^{\E_n} \simeq \Barc\circ \free_{\E_{n-1}}^{\E_n}\circ \free^{\E_{n-1}} \simeq \Sigma_{\E_{n-1}} \circ \free^{\E_{n-1}} \simeq \free^{\E_{n-1}}\circ \Sigma.
\]
Here the second equivalence is a consequence of \cref{thm:suspensionEn} again and the others are clear.

The left adjoints in the diagram above allow us to view all the participating $\infty$-categories as monadic over $\Sp$ and the corresponding diagram of monads on $\Sp$ is of the form
\[\begin{tikzcd}
	\free_{\E_n} \ar{r}{\iota_n} \ar{d}\ar{r} \ar{dr}{\sigma} & \free_{\E_{n+1}} \ar{d} \\
	\Omega\free_{\E_{n-1}}\Sigma  \ar{r}{\iota_{n-1}} & \Omega\free_{\E_n} \Sigma.
\end{tikzcd}\]
By our earlier discussion on operadic suspension, we can identify the monads on the bottom row with $\free_{S\E_{n-1}}$ and $\free_{S\E_n}$ respectively. The conclusion of the theorem now follows from \cref{thm:SymffEuler} and the fact that all operads involved have nilpotent Euler classes, cf.\ \cref{prop:EnoperadEuler} and \cref{lem:closureSSeqEuler}.
\end{proof}

\begin{remark}
As indicated in \cref{Remark:goodwillie-calculus-differentiate-monads} one could also appeal to a monoidal structure on the Goodwillie derivative $\partial_* \colon \End(\Sp) \to \SSeq(\Sp)$ to obtain the square of operads of \cref{thm:factorizationsuspension} from the square of monads at the end of the proof we just gave. The argument we have given above circumvents the issue that this monoidal structure is not yet established and gives a sharper conclusion, by the fully faithfulness established in \cref{thm:SymffEuler}.
\end{remark}

\begin{remark}
\label{rmk:wrongway}
The `wrong-way morphism' $\beta_n\colon \E_n \to S\E_{n-1}$ arising from the bar construction can be described in many ways. First of all, maps like this were already considered by May \cite{may}. Ching--Salvatore \cite{chingsalvatore} also describe a map of operads $\E_n \to S\E_{n-1}$ arising from the Koszul self-duality of $\E_n$. To be precise, they construct a functor
\[
K\colon \Operad(\Sp) \to \Operad(\Sp)^{\mathrm{op}}
\]
forming the `Koszul dual' $K\mathcal{O}$ of an operad $\mathcal{O}$. It is given by the levelwise Spanier--Whitehead dual of the operadic bar construction. Then they prove that there is an equivalence of operads in $\Sp$ of the form
\[
K\E_n \cong S^{-n}\E_n. 
\]
Under this equivalence (and the analogous one for $n-1$) the evident inclusion $\iota_{n-1}\colon \E_{n-1} \to \E_n$ produces a map of operads $\E_n \to S\E_{n-1}$; Ching--Salvatore include an explicit description of this map in \cite{chingsalvatore}. This morphism should be equivalent to the $\beta_n$ we have constructed, although we will not pursue the matter here. It is interesting to note that Ching--Salvatore's map already exists \emph{unstably}, as a map of operads in pointed spaces. Since we have only dealt with the operadic suspension and the suspension morphism for $\E_n$ in the stable case, our methods cannot address this unstable version.

Moreover, as indicated in the introduction, \cite[Theorem C]{Fresse} also provides a square of operads as in \cref{thm:factorizationsuspension} after tensoring with an ordinary commutative ring $R$, based on explicit point-set models. Again, our square (after tensoring with $R$) should be equivalent to his, but we will not pursue the matter here.
\end{remark}

\subsection{The suspension morphism for operads in pointed spaces}
\label{subsec:unstableEn}

We have discussed the operadic suspension in the stable case, i.e.\ for operads in $\Sp$, but in fact it should already exist in the context of operads in the symmetric monoidal $\infty$-category of pointed spaces with smash product, or generally in pointed symmetric monoidal $\infty$-categories for which the tensor product is compatible with colimits in each variable. As indicated earlier, for the case of pointed spaces there are several versions of this in the literature. Arone--Kankaanrinta \cite{aronekankaanrinta} construct a \emph{sphere operad} $S^{\rho}$, of which the terms $S^{\rho}(n)$ are equivalent to the representation spheres $S^{\rho_n}$ and the composition maps
\[
S^{\rho}(n) \wedge S^{\rho}(k_1) \wedge \cdots \wedge S^{\rho}(k_n)  \to S^{\rho}(k_1 + \cdots + k_n)
\]
are homotopy equivalences (or even homeomorphisms). For a general non-unital operad in pointed spaces, one can then define its operadic suspension by taking the levelwise smash product with this sphere operad. An alternative construction of a sphere operad and the corresponding operadic suspension is given by Ching--Salvatore \cite{chingsalvatore}.

In order to compare these constructions to each other and to the operadic suspension we have discussed here it is desirable to characterize the operadic suspension by universal properties, rather than focus on any specific construction of it. In this brief section we point out what such properties might be and invite the reader to take this up. We speculate that our factorizations of the suspension morphism $\sigma\colon \Sigma^\infty_+ \E_n^{\mathrm{nu}} \to S\Sigma^\infty_+ \E_n^{\mathrm{nu}}$ of \cref{thm:factorizationsuspension} already exist for the suspension morphism $\sigma\colon (\E_n^{\mathrm{nu}})_+ \to S(\E_{n-1}^{\mathrm{nu}})_+$ of the (unstable) pointed $\E_n$-operad, which is essentially obtained by means of the Euler classes. If true, these factorizations should then provide alternative proofs of \cref{thm:suspensionEn} (on suspensions of $\E_n$-algebras) and \cref{cor:loopsEn} (on loops of $\E_n$-algebras).

We will now describe the desired properties of the operadic suspension. Consider a presentable pointed symmetric monoidal $\infty$-category $(\C,\wedge)$ and assume that the tensor product commutes with colimits in each variable separately. Let $\O$ be a non-unital operad in $\C$. Then there should be a suspended operad $S\O$ with terms 
\[ S\O(n) = \O(n) \wedge S^{{\rho}_n}\]
together with a canonical suspension morphism $\sigma \colon \O \to S \O$ induced from the Euler classes $S^0 \to S^{{\rho}_n}$. These data should satisfy the following property:
\begin{itemize}
\item[(F)] There is a canonical equivalence in $\End(\C)$ between $\Sigma \free_{S\O}$ and $\free_\O  \Sigma$. Its adjoint map $\free_{S\O} \to \Omega \free_\O \Sigma$ refines to a map of monads, making the composite
\[ \free_{\O} \xrightarrow{\free_{\sigma}} \free_{S\O} \lto \Omega \free_{\O}\Sigma \]
equivalent to the suspension morphism of the reduced monad $\free_\O$.
\end{itemize}
Let us briefly explain the relation between the factorization above and the suspension-loop adjunction on the $\infty$-category $\Alg_{\mathcal{O}}(\C)$ itself. Consider the following right adjoint functors:
\[
\begin{tikzcd}
\Alg_{\O}(\C) \xrightarrow{\,\Omega\,} \Alg_{\O}(\C) \xrightarrow{\fgt} \C.
\end{tikzcd}
\]
The monad on $\C$ associated with the composed adjunction is $\Omega \free_{\O}\Sigma$. The composite of right adjoints must therefore factor over a right adjoint functor $\gamma^*\colon \Alg_{\O}(\C) \to \Alg_{\Omega \free_{\O}\Sigma}(\C)$. Combining this with property (F) above allows us to construct a diagram of right adjoint functors
\[
\begin{tikzcd}
\Alg_{\O}(\C) \ar{r}{\gamma^*}\ar{d}{\fgt}\ar[bend left = 20]{rrr}{\Omega} & \Alg_{\Omega \free_{\O}\Sigma}(\C)\ar{d}[swap]{\fgt}\ar{r} & \Alg_{S\O}(\C) \ar{r}{\sigma^*} \ar{dl}[swap]{\fgt} & \Alg_{\O}(\C) \ar[bend left=10]{dll}{\fgt}	 \\
\C \ar{r}{\Omega} & \C. &&
\end{tikzcd}
\]
In particular, this shows that the loops functor $\Alg_{\O}(\C) \xrightarrow{\Omega} \Alg_{\O}(\C)$ factors over the restriction $\sigma^*\colon \Alg_{S\O}(\C) \to \Alg_{\O}(\C)$ along the suspension morphism. Dually, the suspension functor of $\Alg_{\O}(\C)$ factors over the pushforward functor $\sigma_!$. We note that in the case where $\C$ is stable, the first two arrows on the top row of the diagram are equivalences of $\infty$-categories and their composite coincides with the shift we have considered earlier. However, in the general case these need not be equivalences. For general $\C$, it would be desirable if $S\O$, together with its suspension morphism $\sigma \colon \O \to S\O$ were determined by property (F) described above, or to give some other additional characterizing properties. We invite the interested reader to take this up.

\section{Triviality of spheres}
\label{sec:trivspheres}

In this section we discuss some examples where our results can be applied and we investigate in what sense our results are optimal.
\subsection{Preliminaries}
For this section, let us fix an $\E_{n+1}$-ring spectrum $A$, so that $\Mod(A)$ is a presentably $\E_{n}$-monoidal $\infty$-category. We write $\Alg_{\E_{n}}(A)$ for $\Alg_{\E_{n}}(\Mod(A))$ and likewise for the augmented and non-unital versions. 
In what follows, an object of $\Alg_{\E_n}^\aug(A)$ is called $\E_n$-trivial over $A$ if it is equivalent to $\tr_{\E_n}$ of its underlying augmented $\E_0$-algebra. For a space $X$, we write $C^*(X;A)$ for the ``cochains of $X$ with coefficients in $A$'', i.e.\ for the $X$-indexed limit of the constant diagram on $A$. A choice of a basepoint on $X$ makes $C^*(X;A)$ into an object of $\Alg_{\E_n}^\aug(A)$. 
As a consequence of \cref{prop:nfoldsusploops}, we have the following result.
\begin{theorem}\label{thm:formality-of-spheres}
Let $X$ be a pointed space and $n \geq 0$. Then $C^*(\Sigma^n X;A)$ is $\E_n$-trivial over $A$. In particular, $C^*(S^n;A)$ is $\E_n$-trivial over $A$.
\end{theorem}
\begin{proof}
The functor $C^*(-;A) \colon (\Spc_*)^\op \to \Alg_{\E_n}^\aug(A)$ preserves pullbacks and hence sends $\Sigma^n X$ to $\Omega^n C^*(X;A)$. The result then follows from \cref{prop:nfoldsusploops}.
\end{proof}

Let us assume for a moment that $A$ is a discrete commutative ring and let us also denote by $A$ its associated Eilenberg--Mac Lane spectrum. Then $H^*(\Sigma^n X;A)$ is also canonically an object of $\Alg_{\E_\infty}^\aug(A)$, via the composite of lax symmetric monoidal functors 
\begin{equation}\label{display:H-as-commutative-ring}
\Mod(A) \xrightarrow{\pi_*} \Mod(A)^\heartsuit_\gr \xrightarrow{\,H\,} \Mod(A).
\end{equation}
Here, $\Mod(A)^\heartsuit_\gr$ denotes the ordinary symmetric monoidal category of graded discrete $A$-modules, whereas $\Mod(A)$ is equivalent to $\D(A)$, the unbounded derived category of $A$, and the second functor in the above display is the graded Eilenberg--Mac Lane functor. 

\begin{definition}
Let $X$ be a space and $A$ a commutative ring. We say that $X$ is \emph{$\E_n$-formal over $A$} if there exists an equivalence $C^*(X;A) \simeq H^*(X;A)$ in $\Alg_{\E_n}^\aug(A)$. 
\end{definition}

\begin{remark}
The formal spaces from rational homotopy theory are, in the above notation, the $\E_\infty$-formal spaces over $\Q$. This uses the equivalence between commutative differential graded algebras over $\Q$ and $\Alg_{\E_\infty}(\D(\Q))$ \cite[Prop.\ 7.1.4.11]{HA}. Surprisingly, a space which is $\E_1$-formal over $\Q$ is already $\E_\infty$-formal over $\Q$ \cite{saleh, CPRNW} so that the notion of $\E_k$-formality over $\Q$ is independent of $k$ once $k\geq1$. This is very special to fields of characteristic $0$. Indeed, as we shall see below, the same statement fails for non-rational rings. 
\end{remark}

Our main result implies the following formality result.
\begin{theorem}\label{Thm:formality}
Let $X$ be a pointed space and $n\geq0$. If there exists an equivalence $\tilde C^*(X;A) \simeq \tilde H^*(X;A)$ in $\Mod(A)$, then $\Sigma^n X$ is $\E_n$-formal over $A$.
\end{theorem}
\begin{proof}
We have observed above that $C^*(\Sigma^nX;A) \simeq \tr_{\E_n}(\tilde C^*(X;A))$. Hence it suffices to observe that also $H^*(\Sigma^n X;A) \simeq \tr_{\E_n}(\tilde H^*(X;A))$. This follows from the fact that the composite \eqref{display:H-as-commutative-ring} commutes with the suspension and loop functors.
\end{proof}

\begin{example}\label{ex:suspensions-Q}
It is well-known in rational homotopy theory that suspensions are formal. One way to prove this is to show that any suspension is rationally equivalent to a wedge of spheres, and then to separately show that spheres are formal, and that wedges of formal spaces are formal, see e.g.\ \cite[Theorem 10.28]{Berglund} under additional finiteness assumptions. \cref{Thm:formality} gives the following non-computational proof of the formality of suspensions. Assume that $A$ is a commutative ring of global dimension $\leq 1$, e.g.\ a Dedekind domain or a field. It is well-known that for any object $M \in \Mod(A)$, there exists an equivalence $M \simeq \oplus_{n} \Sigma^n \pi_n(M)$, see e.g.\ \cite[proof of Cor.\ 3.8]{HLN} which applies verbatim to the case at hand. The right hand side is the value of the composite \eqref{display:H-as-commutative-ring} applied to $M$. In particular $\Sigma X$ is $\E_1$-formal over $A$, and hence as indicated above, $\E_\infty$-formal over $\Q$.
\end{example}

\begin{example}
Suppose that $X$ is of finite $\Z$-homological type and $A$ is a commutative ring. Then the canonical map $C^*(X;\Z)\otimes_\Z A \to C^*(X;A)$ is an equivalence. The same statement for cohomology in place if cochains is, of course, not correct in general. However, it is true if the cohomology of $X$ is in addition free as a graded abelian group. In this case, the same is true for $\Sigma^n X$ and reasoning as in \cref{ex:suspensions-Q} one deduces that $\Sigma^n X$ is $\E_n$-formal over any commutative ring $A$. Examples of such $X$ include spheres, projective spaces over $\mathbb{C}$ and $\mathbb{H}$, loop spaces of spheres, free loop spaces of odd dimensional spheres, as well as products of such spaces.
\end{example}

\begin{remark}
There are many examples of formal spaces whose cochain algebras are not trivial algebras, for instance compact K\"ahler manifolds different from $S^2$. Our methods cannot be applied in this situation. 
\end{remark}

We single out the most basic cases of suspensions in the following corollary.
\begin{corollary}
For $n\geq0$, the $n$-sphere $S^n$ is $\E_n$-formal over any commutative ring.
\end{corollary}

Coming back to interpreting formality of spheres in terms of triviality, one may wonder whether our results are suboptimal in the sense that $n$-fold loop objects are even trivial as $\E_{n+1}$-algebras. The above case of cochains on suspensions gives a good test case for such questions. 
\begin{observation}\label{rem:detecting-non-triviality}
A consequence of \cref{lemma:trivial-and-base-change} in the present context is the following. Suppose $A$ and $B$ are $\E_{k+1}$-algebras in $\Sp$. If $C^*(S^n;A)$ is $\E_k$-trivial over $A$, then $C^*(S^n;A \otimes B)$ is $\E_k$-trivial over $A \otimes B$ and over $B$. Indeed, this follows by applying \cref{lemma:trivial-and-base-change} to the extension of scalars functor $- \otimes B\colon \Alg_{\E_n}^\aug(\Mod(A)) \to \Alg_{\E_n}^\aug(\Mod(A\otimes B))$ as well as to the restriction of scalars functor $\Alg_{\E_n}^\aug(\Mod(A\otimes B)) \to \Alg_{\E_n}^\aug(\Mod(B))$.
\end{observation}

The following result shows that, in general, the conclusion of \cref{thm:formality-of-spheres} is not sharp. 
\begin{lemma}\label{lemma:rational-case}
Let $A$ be a rational $\E_{m+1}$-ring spectrum where $m\geq n+1 \geq 2$. Then $C^*(S^n;A)$ is trivial as an $\E_{m}$-$A$-algebra. 
\end{lemma}
\begin{proof}
It follows from \cref{ex:suspensions-Q} that $C^*(S^n;\Q)$ is in particular $\E_{m}$-trivial over $\Q$. Since the unit of $A$ factors via an $\E_{m+1}$-map $\Q \to A$, the lemma follows from \cref{rem:detecting-non-triviality}.
\end{proof}

We believe that the case of rational ring spectra is, however, the only case in which our result is not sharp. Recall the following conjecture already stated in the introduction:
\begin{conjecture}\label{conj}
Let $n\geq 1$ and let $A$ be an $\E_{n+2}$-algebra in $\Sp$. Then $C^{*}(S^{n};A)$ is $\E_{n+1}$-trivial over $A$ if and only if $A$ is rational.
\end{conjecture}

The rest of this section is devoted to an inspection of this conjecture. We offer two approaches, one based on a study of the operadic structure maps and one based on power operations. These two approaches yield the following two results, respectively.\footnote{In fact, we use the operadic approach to conclude only the version of statement \ref{item:thm2} where $A$ is connective.}
\begin{theorem}\label{thm:conj}
The \hyperlink{conj}{Triviality Conjecture} holds 
\begin{enumerate}
\item\label{item:thm1} when $n\leq 2$,
\item\label{item:thm2} when $A$ is bounded below,
\item\label{item:thm3} when $A$ is the underlying $\E_{n+2}$-algebra of an $\E_\infty$-algebra. 
\end{enumerate}
\end{theorem}

\begin{theorem}\label{thm:conj2}
Suppose $A$ is an $\E_{n+2}$-algebra in $\Sp$ for which there exists a prime $p$ such that $A \otimes \F_p$ or $A \otimes \KU/p$ is non-zero. Then $C^*(S^n;A)$ is not $\E_{n+1}$-trivial over $A$. In particular, the \hyperlink{conj}{Triviality Conjecture} holds for such $A$.
\end{theorem}

We will prove \cref{thm:conj} in \cref{subsec:proofoconjn12} and \cref{thm:conj2} in \cref{subsec:powerops}.
The two results are not unrelated. For instance, we record the following implication.
\begin{lemma}\label{lemma:implication}
\cref{thm:conj2} implies Theorems~\ref{thm:conj}~\ref{item:thm2} and \ref{thm:conj}~\ref{item:thm3}.
\end{lemma}
\begin{proof}
First, suppose $A$ is bounded below and not rational. Then there exists a prime $p$ such that $A \otimes \F_p$ is non-zero, so \cref{thm:conj2} implies that $C^*(S^n;A)$ is not $\E_{n+1}$-trivial over $A$, showing the first implication. To prove the second, suppose $A$ underlies an $\E_\infty$-algebra and that $A \otimes \F_p = 0 = A \otimes \KU/p$ for all primes $p$. We need to show that $A$ is rational. To that end,  as a consequence of \cite{Hahn}, we find that $L_{K(n,p)}A = 0$ for all $0 < n \leq \infty$ and all primes. Consequently, $K(n,p) \otimes A \otimes \End(\bS/p) = 0$ for all $p$ and all $0\leq n \leq \infty$.\footnote{Thanks to Achim Krause for reminding us of this fact.} The nilpotence theorem \cite[Theorem 1]{HS} then implies that the ring spectrum $A \otimes \End(\bS/p)$ is null, implying that $p$ is invertible in $A$ for all primes $p$; see \cite[Lemma 2.3]{LMMT} for a similar argument.
\end{proof}

\subsection{Operadic structure maps}
\label{subsec:proofoconjn12}
In this section, we continue to write $\E_n$ for the non-unital $\E_n$-operad in spectra $\Sigma^\infty_+ \E_n^\nu$ and work towards the proof of \cref{thm:conj}.
Recall that for a spectrum $X$, a natural number $n$, and a prime $p$, we write
\[ D_{p}^{\E_n}(X) = [\E_n(p)\otimes X^{\otimes p}]_{h\Sigma_p}.\]
We will need some relations between these extended powers for particular values of $n$ and $X$. First observe that \cref{thm:factorizationsuspension} yields commutative squares of non-unital operads
\[
\begin{tikzcd}
	S^{-n+1} \E_{n-1} \ar[r] \ar[d] & S^{-n+1} \E_{n+1} \ar[d] && S^{-n+1} \E_{n-1} \ar[r] \ar[d] & S^{-n+1} \E_{\infty} \ar[d] \\
	\E_{0} \ar[r] & \E_{2}, && \E_{0} \ar[r] & \E_{\infty}.
\end{tikzcd}
\]
The horizontal maps are the usual inclusions, the first three vertical maps are $(n-1)$-fold compositions of the `wrong-way maps' $\beta$, and the last vertical morphism $S^{-n+1} \E_{\infty} \to \E_\infty$ is an $(n-1)$-fold composition of suspension maps $\sigma$. If we take free algebras on the spectrum $\bS^{-1}$ and use that $\E_0(p) = 0$, then the squares above induce sequences
\[
\Sigma^n D_p^{\E_{n-1}}(\bS^{-n}) \lto \Sigma^n D_p^{\E_{n+1}}(\bS^{-n}) \lto \Sigma D_p^{\E_2}(\bS^{-1})
\]
and
\[
\Sigma^n D_p^{\E_{n-1}}(\bS^{-n}) \lto \Sigma^n D_p^{\E_{\infty}}(\bS^{-n}) \lto \Sigma D_p^{\E_{\infty}}(\bS^{-1})
\]
equipped with nullhomotopies of the composite maps. Sequences of this form were already analyzed by Kuhn \cite{Kuhn}.

\begin{lemma}
\label{lem:extendedpowercofib}
After $p$-localization, the two sequences above are cofiber sequences. As a consequence the commutative square
\[
\begin{tikzcd}
	\Sigma^n D_p^{\E_{n+1}}(\bS^{-n}) \ar[r] \ar[d] & \Sigma^n D_p^{\E_{\infty}}(\bS^{-n}) \ar[d]  \\
	\Sigma D_p^{\E_{2}}(\bS^{-1}) \ar[r] & \Sigma D_p^{\E_{\infty}}(\bS^{-1}),
\end{tikzcd}
\]
again obtained from \cref{thm:factorizationsuspension}, is a pushout.
\end{lemma} 
\begin{proof}
If the two sequences are indeed cofiber sequences then the fibers of the vertical maps in the square agree, implying that it is indeed a pushout. At $p=2$ it is rather straightforward to see that the sequences are cofiber sequences by explicitly identifying them with
\[
\RP_{-n}^{-2} \lto \RP_{-n}^{0} \lto \RP_{-1}^{0} 
\]
and
\[
\RP_{-n}^{-2} \lto \RP_{-n}^{\infty} \lto \RP_{-1}^{\infty}.
\]
For an odd prime $p$ we argue using homology. Since we have already specified nullhomotopies for the composite maps, it will suffice to show that taking homology with $\F_p$-coefficients yields short exact sequences of graded vector spaces. Write $\iota_n \in H_{-n}(\bS^{-n};\F_p)$ for the fundamental class. According to \cite[\S 16]{GKRW} the homology of $\Sigma^n D_p^{\E_{\infty}}(\bS^{-n})$ has a basis consisting of the classes
\begin{eqnarray*}
Q^s(\iota_n) \in H_{2s(p-1)}(\Sigma^n D_p^{\E_{\infty}}(\bS^{-n});\F_p) && \text{for } 2s \geq -n, \\
\beta Q^s(\iota_n) \in H_{2s(p-1)-1}(\Sigma^n D_p^{\E_{\infty}}(\bS^{-n});\F_p) && \text{for } 2s > -n,
\end{eqnarray*}
with $Q^s$ and $\beta Q^s$ the Dyer--Lashof operations. In particular, for $n = 1$ we have classes $Q^s(\iota_1)$ and $\beta Q^s(\iota_1)$ for $s \geq 0$ and the suspension morphism $\Sigma^n D_p^{\E_{\infty}}(\bS^{-n}) \to \Sigma D_p^{\E_{\infty}}(\bS^{-1})$ maps $Q^s(\iota_n)$ to $Q^s(\iota_1)$ and similarly for $\beta Q^s$. Indeed, the suspension arises from an $\E_\infty$-map $\free_{\E_\infty}(\bS^{-n}) \to \Omega^{n-1} \free_{\E_\infty}(\bS^{-1})$ and is therefore compatible with Dyer--Lashof operations. Thus the induced map in homology is a surjection with kernel spanned by the classes $Q^s(\iota_n)$ and $\beta Q^s(\iota_n)$ with $s \leq -1$. This submodule is precisely the image of the homology of $\Sigma^n D_p^{\E_{n-1}}(\bS^{-n})$ (cf.\ \cite[\S 16]{GKRW}).\footnote{Note that the top two operations $Q^{-1}$ and $\beta Q^{-1}$ are labelled $\xi$ and $\zeta$ in loc.\ cit.} Here we have used that the Browder bracket vanishes for degree reasons. Indeed, the fundamental class $\iota_n$ has degree $-n$, whereas the bracket has degree $n$. These have the same parity, so that graded antisymmetry of the bracket implies ${[\iota_n,\iota_n]} = 0$. Finally, we conclude that the second sequence is indeed short exact on homology.

To deduce that the first sequence is a cofiber sequence, we consider the square of the lemma again. We have already analyzed the homology of the vertical map on the right. The homology of the left-hand map is obtained by restricting to the submodules on classes $Q^s(\iota)$ and $\beta Q^s(\iota)$ for which $s \leq 0$. The resulting sequence 
\[
H_*(\Sigma^n D_p^{\E_{n-1}}(\bS^{-n})) \lto H_*(\Sigma^n D_p^{\E_{n+1}}(\bS^{-n})) \lto H_*(\Sigma D_p^{\E_{2}}(\bS^{-1}))
\]
is still short exact, completing the proof.
\end{proof}

\begin{remark}
The two cofiber sequences of \cref{lem:extendedpowercofib} are special cases of the statement of \cite[Prop.\ 1.3]{Kuhn}, but the odd primary case in loc.\ cit.\ is in general not correct as stated. We will only need the special cases that we established above, and hence only give a proof of these.
\end{remark}

\begin{remark}
\label{rmk:Bockstein}
It is useful to observe that the $\F_p$-homology of $\Sigma D_p^{\E_{2}}(\bS^{-1})$ is spanned by the classes $Q^0(\iota_1)$ and $\beta Q^0(\iota_1)$, with the Bockstein acting as indicated by the notation. It follows that $\Sigma D_p^{\E_{2}}(\bS^{-1})$ is $p$-locally equivalent to the Moore spectrum $\bS^{-1}/p$ and that the wrong-way map
\[
\beta\colon \Sigma D_p^{\E_{2}}(\bS^{-1}) \lto D_p^{\E_{1}}(\bS^0) \cong \bS^0
\]
may be identified with projection to the top cell, i.e., the integral Bockstein map $\bS^{-1}/p \to \bS^0$.
\end{remark}

\begin{remark}
\label{rmk:transfer}
A variation of the previous remark is the following. The homology of $\Sigma D_p^{\E_{\infty}}(\bS^{-1})$ is spanned by the classes $Q^s(\iota)$ and $\beta Q^s(\iota)$ for $s \geq 0$, whereas the homology of $D_p^{\E_\infty}(\bS^0) = \bS^0_{h\Sigma_p}$ has all the same classes except for the one class $\beta Q^0(\iota)$ in degree -1. It follows that there is a $p$-local cofiber sequence
\[
\begin{tikzcd}
\bS^{-1} \lto \Sigma D_p^{\E_{\infty}}(\bS^{-1}) \lto \bS^0_{h\Sigma_p} 
\end{tikzcd}
\]
whose connecting map $\bS^0_{h\Sigma_p} \to \bS^0$ can be identified with the transfer for $\Sigma_p$.
\end{remark}

If $X$ is a (non-unital) $\E_{n+1}$-algebra in $\Sp$, or more generally in $\Mod(A)$ for some $\E_{n+2}$-ring spectrum $A$, there are structure maps 
\[ \mu_{p}^{n+1} \colon D_{p}^{\E_{n+1}}(X) \lto X.\]
We shall be interested in these maps in the case where $X$ is the augmentation ideal of $C^*(S^n;A)$ for $A$ such an $\E_{n+2}$-algebra. In this case, the structure maps are determined by the  composition
\begin{equation} \label{eq:structure-maps}
\mu_{p}^{n+1} \colon D_{p}^{\E_{n+1}}(\bS^{-n}) \lto \bS^{-n} \lto \Sigma^{-n} A, 
\end{equation}
where the first is the structure map in case $A= \bS$ and the second is the $n$-fold desuspension of the unit of the ring spectrum $A$. We can shift the first map to obtain a map
\[
\Sigma^n D_{p}^{\E_{n+1}}(\bS^{-n}) \lto \bS^0
\]
which we still somewhat abusively denote $\mu_p^{n+1}$. The $n$-fold wrong-way map $\beta^n$ of \cref{thm:factorizationsuspension} provides a commutative square
\[
\begin{tikzcd}
\Sigma^n D_{p}^{\E_{n+1}}(\bS^{-n}) \ar{r}{\mu_p^{n+1}}\ar{d}{\beta^n} & \bS^0 \ar[equal]{d} \\
D_p^{\E^1}(\bS^0) \ar{r}{\mu_p^1} & \bS^0. 
\end{tikzcd}
\]
The structure map $\mu_p^{1}$ is an isomorphism, since it arises from the cochain algebra of $S^0$. The square therefore allows us to identify $\mu_p^{n+1}$ with $\beta^n$, up to isomorphism. We write $C(\mu_{p}^{n+1})$ for the cofiber of $\mu_p^{n+1}$. In particular, this cofiber  comes equipped with a map $\gamma\colon \bS^0 \to C(\mu_{p}^{n+1})$.  The evident factorization
\[
\Sigma^n D_{p}^{\E_{n+1}}(\bS^{-n}) \xrightarrow{\beta^{n-1}} \Sigma D_{p}^{\E_{2}}(\bS^{-1}) \stackrel{\beta}{\lto} \bS^0 
\]
of $\beta^n$ provides a cofiber sequence
\[
C(\beta^{n-1}) \lto C(\mu_p^{n+1}) \lto C(\beta).
\]
By \cref{lem:extendedpowercofib}, we have $C(\beta^{n-1}) \simeq \Sigma^{n+1} D_p^{\E_{n-1}}(\bS^{-n})$ when $n\geq 2$ and $0$ for $n=1$. Moreover, as observed in \cref{rmk:Bockstein} we may identify $C(\beta)$ with $\bS^0$. The key to our proof will be the following result.

\begin{proposition}\label{lemma:operad-maps}
Let $p$ be a prime and $n \geq 1$. Then after $p$-localization, the cofiber sequence
\[
C(\beta^{n-1}) \lto C(\mu_p^{n+1}) \lto \bS^0
\]
described above is canonically split. The induced map $\bS^0 \xrightarrow{\gamma} C(\mu_p^{n+1}) \simeq \bS^0 \oplus C(\beta^{n-1})$ is then given by $(p,\theta_n)$ for some $\theta_n \in \pi_0(C(\beta^{n-1}))$.
\end{proposition}
\begin{proof}
The case $p=2$ is rather straightforward to deal with directly, so let us do it first. In this case the sequence of the lemma can be identified with the right column in the following diagram, in which all rows and columns are cofiber sequences:
\[
\begin{tikzcd}
\RP_{-n}^0 \ar{r}{\beta^{n-1}}\ar[equal]{d} & \RP_{-1}^0 \ar{r}\ar{d} & \Sigma\RP_{-n}^{-2} \ar{d} \\ 
\RP_{-n}^0 \ar{r}{\mu_p^{n+1}}\ar{d} & \bS^0 \ar{r}\ar{d}{2} & \Sigma\RP_{-n}^{-1} \ar{d} \\
0 \ar{r} & \bS^0 \ar[equal]{r} & \bS^0.
\end{tikzcd}
\]
The spectrum $\Sigma\RP_{-n}^{-1}$ is the Spanier--Whitehead dual of $\RP^n_+$. The bottom cell of the latter splits off; by duality, the projection $\Sigma\RP_n^{-1} \to \bS^0$ to the top cell admits a splitting as well. Furthermore, the diagram shows that the horizontal map $\bS^0 \to \Sigma\RP_n^{-1}$ is indeed of degree 2 onto that top cell. 

We proceed to the case of an odd prime $p$. To prepare, consider the inverse limit $\varprojlim_n \Sigma^n D_p^{\E_\infty}(\bS^{-n})$. By construction this is the evaluation of the functor
\[
F(X) := \varprojlim_n \Sigma^n D_p^{\E_\infty}(\Omega^n X)
\]
at $X = \bS^0$. The functor $F$ is the universal exact (i.e., 1-excisive) approximation \emph{from the left} of the functor $D_p^{\E_\infty}$ and hence the subject of dual Goodwillie calculus. We refer to $F$ as the \emph{dual derivative} of $D_p^{\E_\infty}$. Dual to the usual case, where an $n$-homogeneous functor has trivial linear approximation from the right, an $n$-cohomogeneous functor such as $X \mapsto (X^{\otimes p})^{h\Sigma_p}$ has no linear approximation from the left. Consequently, the fiber sequence 
\begin{equation}\label{dual-derivative}
\Omega (X^{\otimes p})^{t\Sigma_p} \lto (X^{\otimes p})_{h\Sigma_p} \lto (X^{\otimes p})^{h\Sigma_p} 
\end{equation}
shows that the first map produces an equivalence on dual derivatives. Moreover, the first term is an exact functor of $X$.\footnote{Again, we recall that we are working in $p$-local spectra. The decisive point is that in $p$-local spectra, $(-)^{t\Sigma_p}$ vanishes on (additively) $C_p$-induced objects.} Hence the sequence exhibits the first term as the dual derivative of $D_p^{\E_\infty}$. We conclude that $\varprojlim_n \Sigma^n D_p^{\E_\infty}(\bS^{-n}) \cong (\bS^{-1})^{t\Sigma_p}$.

Now take the inverse limit over $n$ of the squares of \cref{lem:extendedpowercofib} to obtain a pushout square
\[
\begin{tikzcd}
	\varprojlim_n \Sigma^n D_p^{\E_{n+1}}(\bS^{-n}) \ar[r] \ar[d] & (\bS^{-1})^{t\Sigma_p} \ar[d]  \\
	\Sigma D_p^{\E_{2}}(\bS^{-1}) \ar[r] & \Sigma D_p^{\E_{\infty}}(\bS^{-1}).
\end{tikzcd}
\]
It admits a map to the pushout square
\[
\begin{tikzcd}
D_p^{\E_1}(\bS^0) = \bS^0 \ar{r}\ar{d}{=} & D_p^{\E_\infty}(\bS^0) = \bS^0_{h\Sigma_p} \ar{d}{=} \\
\bS^0 \ar{r} & \bS^0_{h\Sigma_p}
\end{tikzcd}
\]
by using the limit of the $n$-fold wrong-way maps $\beta^n$ (resp. the $n$-fold suspensions $\sigma^n$) in the top left (resp. the top right) and the map $\beta$ (resp. $\sigma$) in the lower left (resp. the lower right) corners. Taking the cofiber of this map of squares yields a further pushout square
\[
\begin{tikzcd}
\varprojlim_n C(\mu_p^{n+1}) \ar{r}\ar{d} & (\bS^0)^{h\Sigma_p} \ar{d} \\
\bS^0 \ar{r}{\cong} & \bS^0.
\end{tikzcd}
\]
where we have used \cref{rmk:Bockstein} to identify the lower left-hand corner and \cref{rmk:transfer} for the lower right. We now argue that the right vertical map is the canonical projection, which admits a canonical section. Indeed, to do so, we note that by construction, its precomposition with the norm map $\bS_{h\Sigma_p} \to \bS^{h\Sigma_p}$ is the transfer for $\Sigma_p$. Associated to the norm fiber sequence for $\Sigma_p$ acting trivially on $\bS$, we obtain the fiber sequence
\[ \map(\bS^{t\Sigma_p},\bS) \lto \map(\bS^{h\Sigma_p},\bS) \lto \map(\bS_{h\Sigma_p},\bS).\]
Recall that we implicitly work $p$-locally at all times. Then we find $\bS^{t\Sigma_p} \simeq \bS_p$, see e.g.\ \cite[Remark IV.1.6]{NS}. 
Now, $\pi_0(\map(\bS_p,\bS)) = 0$ as follows from the fiber sequence $\tau_{\geq1}\bS \to \bS \to \Z$, the fact that $\tau_{\geq1}\bS$ is $p$-complete and has no $\pi_0$, and that $\mathrm{Hom}(\Z_p,\Z_{(p)}) = 0$. 
Finally, the canonical map $\bS^{h\Sigma_p} \to \bS$ also gives the transfer upon precomposition with the norm, showing the claim.

It remains to verify that the composite of the maps $\bS^0 \to \varprojlim_n C(\mu_p^{n+1}) \to \bS^0$ we have produced is of degree $p$. For this we can inspect the following commutative diagram, analogous to the one for $p=2$ with which we started this proof:
\[
\begin{tikzcd}
\Sigma^n D_p^{\E_{n+1}}(\bS^{-n})  \ar{r}{\beta^{n-1}}\ar[equal]{d} & \Sigma D_p^{\E_{2}}(\bS^{-1}) \ar{r}\ar{d}{\beta} & C(\beta^{n-1}) \ar{d} \\ 
\Sigma^n D_p^{\E_{n+1}}(\bS^{-n})  \ar{r}{\mu_p^{n+1}}\ar{d} & \bS^0 \ar{r}\ar{d}{p} & C(\mu_p^{n+1}) \ar{d} \\
0 \ar{r} & \bS^0 \ar[equal]{r} & \bS^0.
\end{tikzcd}
\]
Here we have again applied \cref{rmk:Bockstein} to describe the middle column.
\end{proof}

We now turn to the proof of \cref{thm:conj} \ref{item:thm1} and \ref{item:thm3}. It is based on the following observation:
The assumption that $C^*(S^n;A)$ is $\E_{n+1}$-trivial over $A$ implies that the composite
\[ \Sigma^n D_p^{\E_{n+1}}(\bS^{-n}) \stackrel{\mu^n_p}{\lto} \bS \lto A \]
is null. When $n\geq 1$, as a consequence of \cref{lemma:operad-maps}, the unit of $A_{(p)}$ then factors as a composite
\[ \bS \xrightarrow{(p,\theta_n)} \bS \oplus \Sigma^{n+1}D_p^{\E_{n-1}}(\bS^{-n}) \xrightarrow{(x+ f)} A_{(p)} \]
for some $x\in \pi_0(A)$ and $f\colon \Sigma^{n+1}D_p^{\E_{n-1}}(\bS^{-n}) \to A_{(p)}$,
giving the relation
\begin{equation}\label{relation-in-pi0} 1 = px +f\theta_n \in \pi_0(A)_{(p)}.\end{equation}

\begin{proof}[Proof of \cref{thm:conj}]
As a warm-up, we will prove part \ref{item:thm2} in case $A$ is connective. In this case, the cell structure of $\Sigma^{n+1}D_p^{\E_{n-1}}(\bS^{-n})$ reveals that there are no maps to $A_{(p)}$, showing that $f=0$. In particular, we deduce from equation \eqref{relation-in-pi0} that $p$ is invertible in $\pi_0(A)_{(p)}$ and hence also in $\pi_0(A)$. As this holds for all primes $p$ we deduce that $A$ is rational as claimed. 

Next we deal with part \ref{item:thm1}. If $n=1$, then the domain of $f$ is contractible, so the same argument as above applies. If $n=2$, then the domain of $f$ is given by $\Sigma^3 D_p^{\E_1}(\bS^{-2})$ which is equivalent to $\bS^{-2p+3}$. In particular, we may think of $\theta_2$ as an element in $\pi_{2p-3}(\bS)$. By Nishida's nilpotence theorem \cite{Nishida}, the map $\theta_2$ is smash-nilpotent. This implies that the element $f\theta_2 \in \pi_0(A)_{(p)}$ is nilpotent in the algebraic sense. In particular, equation \eqref{relation-in-pi0} implies that $p$ is invertible modulo nilpotent elements, and hence is itself invertible. Again, we conclude that $A$ is rational.

To show \ref{item:thm3}, we follow the same strategy and aim to show that $f\theta \in \pi_0(A)_{(p)}$ is nilpotent. To do so, since $A$ is $\E_\infty$, the solution of the May conjecture\footnote{It is worth mentioning that the solution of May's conjecture we use here also makes use of appropriate power operations.} due to Mathew--Naumann--Noel \cite[Theorem B]{MNN} implies that it suffices to show that the image of $f\theta$ under the Hurewicz maps $\pi_0(A)_{(p)} \to H_0(A;\F_p)$ and $\pi_0(A)_{(p)} \to H_0(A;\Q)$ is trivial. This is the case because $\Sigma^{n+1}D_p^{\E_{n-1}}(\bS^{-n}) \otimes k$ is a coconnected spectrum for $k=\Q$ and $k=\F_p$ for all primes $p$. 
\end{proof}

\begin{remark}\label{rem:general-strategy-for-E_infty}
One might wonder whether the element $f\theta_n \in \pi_0(A)$ is always nilpotent, in which case one would deduce from the above argument that $A$ has to be rational. To that end, it would suffice to show that the map $\theta_n$ is smash-nilpotent. However, unlike the cases $n=1,2$ considered above, this is not true for $n\geq 3$ as we explain now, first focusing on the prime $2$. To that end, note that if $\theta_n$ is smash-nilpotent, then $K(1)_*(\theta_n)$ is the zero map. We will argue that in fact this map is non-zero. Indeed, consider the 2-local cofiber sequence
\[ \bS \xrightarrow{(2,\theta_n)} \bS \oplus \Sigma\RP^{-2}_{-n} \lto \Sigma\RP^{0}_{-n}\]
and take its long exact sequence in $\KU$-homology. The first map induces an injective homomorphism and leads to a short exact sequence
\[ 0 \lto \KU_0(\bS) \xrightarrow{(2,\theta_n)} \KU_0(\bS) \oplus \KU_0(\Sigma \RP^{-2}_{-n}) \lto \KU_0(\Sigma \RP^{0}_{-n}) \lto 0. \]
From Adams' calculation of the $K$-theory of projective spaces \cite[Theorem 7.3]{adamsvectorfield} it is easily deduced that this sequence is isomorphic to one of the form
\[ 0 \lto \Z \xrightarrow{(2,\bar{\theta}_n)} \Z \oplus \Z/2^{k}\Z \lto \Z/2^{k+1}\Z \lto 0 \]
with $k= \frac{n-2}{2}$ if $n$ is even and $k = \frac{n-1}{2}$ if $n$ is odd and where $\bar{\theta}_n$ is induced by $\theta_n$. In particular, if $n \geq 3$ then $k \geq 1$ and this sequence implies that $\bar{\theta}_n$ must hit a generator of the group $\Z/2^{k}\Z$. Hence the reduction of $\bar{\theta}_n$ mod 2 is nontrivial, from which it follows that $\theta_n$ induces a nontrivial homomorphism in $\KU/2$-homology, or equivalently in $K(1)$-homology.

A similar argument applies for all other primes $p$. In this case, we have the 
$p$-local cofiber sequence
\[ \bS \xrightarrow{(p,\theta_n)} \bS \oplus \Sigma^{n+1}D_p^{\E_{n-1}}(\bS^{-n}) \lto \Sigma^{n+1}D_p^{\E_{n+1}}(\bS^{-n}) \]
whose long exact sequence in $\KU$-homology again induces a short exact sequence as follows:
\[ 0 \lto \KU_0(\bS) \xrightarrow{(p,\theta_n)} \KU_0(\bS) \oplus \KU_0(\Sigma^{n+1}D_p^{\E_{n-1}}(\bS^{-n})) \lto \KU_0(\Sigma^{n+1}D_p^{\E_{n+1}}(\bS^{-n})) \lto 0. \]
By \cite[Theorem 6.1]{BHK}, this sequence is isomorphic to 
\[ 0 \lto \Z \xrightarrow{(p,\theta_n)} \Z \oplus \Z/p^k\Z \lto \Z/p^{k+1}\Z \lto 0 \]
for $k= \frac{n-2}{2}$ when $n$ is even and $k=\frac{n-1}{2}$ when $n$ is odd. Again, we deduce that $K(1)_*(\theta_n)$ is non-zero, showing that $\theta_n$ is not smash-nilpotent.
For a concrete example, take $n=3$. Then we find that $\theta_3 \colon \bS \to \Sigma^4 D_p^{\E_2}(\bS^{-3}) \simeq \bS^{-2p+2}/p$ and the above calculations imply that $\theta_3$ is given by a lift of a generator of the $p$-torsion of $\pi_{2p-3}(\bS)$ to $\pi_{2p-2}(\bS/p)$.
\end{remark}

\subsection{Power operations}
\label{subsec:powerops}
We separate the two cases of \cref{thm:conj2} in individual propositions.
\begin{proposition}\label{prop:DL-operation}
Let $A$ be an $\E_{n+2}$-algebra in $\Sp$ with $A \otimes \F_p \neq 0$. Then $C^*(S^n;A)$ is not $\E_{n+1}$-trivial over $A$.
\end{proposition}
\begin{proof}
We first recall that any $\E_{n+1}$-$\F_p$-algebra $B$ admits a natural operation $Q^0 \colon \pi_{-n}(B) \to \pi_{-n}(B)$ \cite[Remark 1.5.3]{Lawson}. This operation vanishes in case $B$ is $\E_{n+1}$-trivial over $\F_p$.
Let us then consider the following commutative diagram
\[ \begin{tikzcd}
	\pi_{-n}C^*(S^n;\F_p) \ar[r,"Q^0"] \ar[d] & \pi_{-n}C^*(S^n;\F_p) \ar[d] \\
	\pi_{-n}C^*(S^n;A \otimes \F_p) \ar[r,"Q^0"] & \pi_{-n}C^*(S^n;A\otimes \F_p) 
\end{tikzcd}\]
where the horizontal arrows are given by the operation $Q^0$. It is a fundamental fact about cochains of spaces that the upper horizontal arrow is the identity \cite[Example 1.5.9]{Lawson}. This shows that the lower horizontal map is non-trivial, since the vertical maps are injective. As a consequence of \cref{rem:detecting-non-triviality} we find that $C^*(S^n;A)$ is not trivial over $A$.
\end{proof}
\begin{remark}
We note that the argument for \cref{prop:DL-operation} above shows that, more generally, if $X$ is a pointed space with $H^n(X;\F_p) \neq 0$ for some $n>0$, then $C^*(X;\F_p)$ is not a trivial square-zero extension of $\F_p$ by the augmentation ideal of $C^*(X;\F_p)$ as an $\E_{n+1}$-$\F_p$-algebra.  
Likewise, $X$ is not $\E_{n+1}$-formal over $\F_p$. Indeed, we need to show that $C^*(X;\F_p)$ is not equivalent to $H^*(X;\F_p)$ as $\E_{n+1}$-$\F_p$-algebra. To see this
we may use that $Q^0$ as an operation on $\pi_{-n}(H^*(X;\F_p))$ is trivial (unless $n=0$), whereas it is the identity on $\pi_{-n}(C^*(X;\F_p))$. Hence, in contrast to the case of rational coefficients, with finite coefficients the only $\E_\infty$-formal spaces are the ones whose cohomology is concentrated in degree 0. In particular, the result of \cite{CPRNW} showing that rational $\E_1$-formality implies rational $\E_\infty$-formality does not extend to fields $k$ of positive characteristic. Indeed, pick any space $X$ such that $H^*(X;k)$ is not concentrated in degree $0$. Then $\Sigma X$ is $\E_1$-formal over $k$ but not $\E_\infty$-formal over $k$. See \cite{FC} for further aspects of such results in positive characteristic.
\end{remark}

\begin{proposition}
Let $A$ be an $\E_{n+2}$-algebra in $\Sp$ with $A \otimes \KU/p \neq 0$. Then $C^*(S^n;A)$ is not $\E_{n+1}$-trivial over $A$.
\end{proposition}
\begin{proof}
Again, the proof relies on appropriate power operations for $p$-complete $\KU$-algebras\footnote{In what follows $\KU$ is implicitly meant $p$-complete.} about which we shall use the following facts. 
\begin{enumerate}
\item For $i=0,1$, there are operations $\theta^i \colon \pi_i(A) \to \pi_i(A)$, defined naturally for $p$-complete $\E_\infty$-$\KU$-algebras\footnote{In fact, more generally for $K(1)$-local $\E_\infty$-rings.} satisfying that for any spectrum $X$, the diagram
	\[\begin{tikzcd}
		\KU^{-1}(X) \ar[r,"\theta^1"] \ar[d,"\cong"] & \KU^{-1}(X) \ar[d,"\cong"] \\
		\KU^0(\Sigma X) \ar[r,"\theta^0"] & \KU^0(\Sigma X)
	\end{tikzcd}\]
	commutes. Moreover, for $X=S^{2n}$, we have $\theta^0(\beta^n) = p^{n-1}\beta^n$,\footnote{This follows from the relation $\psi^p(x) -x^p = p\cdot \theta(x)$ for all $x \in \KU^0(X)$.} so also the operation $\theta^1\colon \KU^{-1}(S^{2n-1}) \to \KU^{-1}(S^{2n-1})$ is given by multiplication by $p^{n-1}$.
\item For $i=0,1$, there are operations $\theta_r^i \colon \pi_i(A) \to \pi_i(A/p^r)$, defined naturally for $p$-complete $\E_{2r+1-i}$-$\KU$-algebras, which when $A$ is a $p$-complete $\E_\infty$-$\KU$-algebra, are given by the composite
	\[ \pi_i(A) \xrightarrow{\;\theta^i} \pi_i(A) \lto \pi_i(A/p^r).\]
\end{enumerate}
All these operations (when defined) vanish on trivial $p$-complete $\KU$-algebras. Part (1) above is classical, see e.g.\ \cite{Hopkins}, in particular for $i=0$; for $i=1$ one can use the fiber sequence of \cref{rmk:transfer}. The proof of (2) follows the same strategy, \cite[\S 6.1]{BHK} can be used for the relevant calculations; see also \cite[\S 6.2]{BHK} for discussions about power operations for $p$-complete $\KU$-algebras.

Let us now denote by $A \hat{\otimes} \KU$ the $p$-completion of $A \otimes \KU$. The assumption of the proposition is then equivalent to the statement that $A \hat{\otimes} \KU$ is non-zero. Given this, let $i\in \{0,1\}$ be such that $n= 2r-i$ and consider the commutative diagram
\[\begin{tikzcd}
	\KU^{-i}(S^{2r-i}) \ar[r,"\theta^i_r"] \ar[d] & \KU^{-i}(S^{2r-i})/p^r \ar[d] \\
	(A \hat{\otimes} \KU)^{-i}(S^{2r-i}) \ar[r,"\theta^i_r"] & (A\hat{\otimes} \KU/p^r)^{-i}(S^{2r-i})
\end{tikzcd}\]
whose horizontal maps are induced by $\theta^i_r$.
The vertical maps in the above diagram are induced by the map $\KU \to A \hat{\otimes} \KU$. It follows that the diagram is isomorphic to the diagram
\[ \begin{tikzcd}
	\Z_p \ar[r,"\cdot p^{r-1}"] \ar[d] & \Z/p^r \ar[d] \\
	\pi_0(A\hat{\otimes} \KU) \ar[r] & \pi_0(A\hat{\otimes}\KU/p^r)
\end{tikzcd}\]
whose vertical maps are induced by the unit of $\pi_0(A\hat{\otimes}\KU)$.
We claim that the lower horizontal map is non-zero, showing that $C^*(S^{n};A)$ is not $\E_{n+1}$-trivial over $A$. To see this, it suffices to show $p^{r-1} \in \pi_0(A \hat{\otimes} \KU)/p^r \subseteq \pi_0(A \hat{\otimes} \KU/p^r)$ is non-zero. Aiming for a contradiction, assume that it is zero. Then $p^{r-1} = p^r\cdot x$ for some $x$ in $\pi_0(A\hat{\otimes}\KU)$ and hence $p^{r-1}(1-px) = 0$. Since $n+2 \geq 2r-2$, we find that $A \hat{\otimes} \KU$ is a $K(1,p)$-local $\E_{2r-2}$-algebra, so we deduce from \cite{Hahn2} that $(1-px)$ is nilpotent.\footnote{In case $A$ is $\E_\infty$, this also follows from the solution of May's nilpotence conjecture \cite{MNN} or a $\delta$-ring argument similar to \cite[Lemma 1.5]{Bhatt} using that $\pi_0$ of any $K(1)$-local $\E_\infty$-ring canonically carries the structure of a $\delta$-ring \cite{Hopkins}.} However, it is also invertible since $(A\hat{\otimes} \KU)/(1-px)$ is $p$-complete and $[(A\hat{\otimes} \KU)/(1-px)]/p = [(A\hat{\otimes} \KU)/p]/(1-px) = 0$. This is in contradiction with the assumption that $A \hat{\otimes} \KU$ is non-zero.
\end{proof}

\subsection{Free loop spaces}\label{sec:string}
We finally discuss some applications of our formality results to free loop spaces that might be interesting from the point of view of string topology of spheres.
Let us write $C^*(X)$ for $C^*(X;\Z)$ to simplify notation. We recall that for a connected, nilpotent space $X$ of finite  $\Z$-homological type with finite fundamental group, the free loop space $LX$ as well as $X\times X$ are again of $\Z$-homological finite type. Under this assumption, we have canonical equivalences in $\Alg_{\E_\infty}(\Z)$
\[ C^*(LX) \simeq C^*(X) \otimes_{C^*(X\times X)} C^*(X) \simeq \HH(C^*(X))\]
and that Hochschild homology (relative to $\Z$) is a functor $\HH \colon \Alg_{\E_n}(\Z) \to \Alg_{\E_{n-1}}(\Z)^{\BT}$. The above composite equivalence is then in fact $\T$-equivariant. As a consequence of \cref{thm:formality-of-spheres}, for $n\geq 1$, we obtain an equivalence of $\E_{n-1}$-$\Z$-algebras with $\T$-action
\[ \HH(C^*(\Sigma^n X)) \simeq \HH(H^*(\Sigma^n X)).\]
Indeed, we have canonical equivalences $C^*(\Sigma^nX) \simeq \Omega^n C^*(X)$ and $H^*(\Sigma^n X) \simeq \Omega^n H^*(X)$, so \cref{thm:formality-of-spheres} implies the claim once we observe that there exists an equivalence of $\E_0$-$\Z$-algebras $C^*(X) \simeq H^*(X)$ since the global dimension of $\Z$ is one, see the argument in \cref{ex:suspensions-Q}. In particular, we obtain an equivalence of $\E_{n-1}$-$\Z$-algebras with $\T$-action
\[ C^*(L\Sigma^n X) \simeq \HH(H^*(\Sigma^n X)),\]
providing a reminiscence of formality for $L(\Sigma^n X)$.

As a nice explicit example, for $n\geq 2$ and a commutative ring $A$, we find an equivalence of $\E_{n-1}$-$A$-algebras with $\T$-action
\[ C^*(LS^{n};A) \simeq \HH_A(A[x]/x^2), \quad |x|=-n\]
where $\HH_A$ denotes Hochschild homology relative to $A$. For instance for $A=\Z$,
additively, $\HH(\Z[x]/x^2)$ is very easy to compute, because it is easy to resolve $\Z[x]/x^2$ as a module over $\Z[y,z]/(y^2,z^2)$ (viewed as a graded commutative ring with $|y|=|z|=-n$), where the module action comes from the ring map sending both $y$ and $z$ to $x$, see e.g.\ \cite[E.4.1.8]{Loday}, but the above equivalence also sheds some light on the coherent multiplicative structure in $C^*(LS^n;\Z)$, more precisely its $\E_{n-1}$-algebra structure. 
There are a number of papers surrounding the relation of the string topology of $S^n$ with Hochschild (co)homology of $C^*(S^n)$ or $H^*(S^n)$, in particular for $n=2$ and with coefficients $\F_2$, see e.g.\ \cite{Menichi, PT}. It would be interesting to find a relation between these works and our formality result.

Finally, we mention that in \cite{Naef} the authors show, using the the above relation to the free loop space, that $C^*(S^2;\F_2)$ is not formal as a \emph{framed} $\E_2$-algebra. It is natural to then ask whether $C^*(S^n;k)$ is formal as a framed $\E_n$-algebra if and only if $k$ is rational, but we will not address this here.

\bibliographystyle{amsalpha}
\bibliography{biblio}

\end{document}